\documentclass[reqno]{amsart}

\usepackage{bbm}
\usepackage{amssymb}
\usepackage{enumitem}   
\usepackage{hyperref}

\newcommand*{\mailto}[1]{\href{mailto:#1}{\nolinkurl{#1}}}
\newcommand{\arxiv}[1]{\href{http://arxiv.org/abs/#1}{arXiv: #1}}

\newtheorem{theorem}{Theorem}[section]

\newtheorem{lemma}[theorem]{Lemma}
\newtheorem{proposition}[theorem]{Proposition}
\newtheorem{corollary}[theorem]{Corollary}
\newtheorem{hypothesis}[theorem]{Hypothesis}
\newtheorem{remark}[theorem]{Remark}

\newcommand{\R}{{\mathbb R}}
\newcommand{\N}{{\mathbb N}}

\newcommand{\C}{{\mathbb C}}
\renewcommand{\H}{{\mathbb C}_+}

\newcommand{\spr}[2]{\langle #1 , #2 \rangle}

\newcommand{\E}{\mathrm{e}}
\newcommand{\I}{\mathrm{i}}

\newcommand{\ess}{\mathrm{ess}}

\newcommand{\im}{\mathrm{Im}}

\newcommand{\supp}{\mathrm{supp}}
\newcommand{\indik}{\mathbbm{1}}
\newcommand{\linspan}{\mathrm{span}}

\newcommand{\dom}[1]{\mathrm{dom}\left(#1\right)}
\newcommand{\mul}[1]{\mathrm{mul}\left(#1\right)}
\renewcommand{\ker}[1]{\mathrm{ker}\left(#1\right)}
\newcommand{\ran}[1]{\mathrm{ran}\left(#1\right)}

\newcommand{\be}{\begin{equation}}
\newcommand{\ee}{\end{equation}}

\newcommand{\OO}{\mathcal{O}}
\newcommand{\oo}{o}

\newcommand{\ledot}{\,\cdot\,}
\newcommand{\redot}{\cdot\,}

\newcommand{\cB}{\mathcal{B}}

\newcommand{\cH}{\mathcal{H}}

\newcommand{\loc}{\mathrm{loc}}
\newcommand{\Tloc}{\mathrm{T}_\loc}
\newcommand{\Tmax}{\mathrm{T}_{\mathrm{max}}}
\newcommand{\Tmin}{\mathrm{T}_{\mathrm{min}}}
\newcommand{\Tmaxpm}{\mathrm{T}_{\mathrm{max},\pm}}
\newcommand{\Tminpm}{\mathrm{T}_{\mathrm{min},\pm}}
\newcommand{\T}{\mathrm{T}}

\newcommand{\Hloc}{\cH_\loc(\R)}
\newcommand{\Hab}{\cH([a,b))}
\newcommand{\Hoab}{\cH_0([a,b))}

\newcommand{\HR}{\cH(\R)}

\newcommand{\dip}{\upsilon}
\newcommand{\Lmu}{L^2(\R;\mu)}

\newcommand{\F}{\mathcal{F}}
\newcommand{\G}{\mathcal{G}}
\newcommand{\M}{\mathrm{M}}
\newcommand{\D}{\mathcal{D}}
\newcommand{\Pe}{\mathrm{L}}
\newcommand{\U}{\mathcal{F}}

\newcommand{\f}{{\bf f}}
\newcommand{\g}{{\bf g}}
\newcommand{\h}{{\bf h}}

\renewcommand{\labelenumi}{(\roman{enumi})}
\renewcommand{\theenumi}{\roman{enumi}}
\numberwithin{equation}{section}

\begin{document}

\title[Quadratic pencils associated with the Camassa--Holm flow]{Quadratic operator pencils associated with the~conservative~Camassa--Holm flow}

\author[J.\ Eckhardt]{Jonathan Eckhardt}
\address{School of Computer Science \& Informatics\\ Cardiff University\\ Queen's Buildings \\ 
5 The Parade\\ Roath \\ Cardiff CF24 3AA\\ Wales \\ UK}
\email{\mailto{j.eckhardt@cs.cardiff.ac.uk}}

\author[A.\ Kostenko]{Aleksey Kostenko}
\address{Faculty of Mathematics\\ University of Vienna\\
Oskar-Morgenstern-Platz 1\\ 1090 Wien\\ Austria}
\email{\mailto{duzer80@gmail.com}; \mailto{Oleksiy.Kostenko@univie.ac.at}}

\thanks{{\it Research supported by the Austrian Science Fund (FWF) under Grants No.\ J3455 and P26060}}

\keywords{Sturm--Liouville problems, quadratic operator pencils, (inverse) spectral theory}
\subjclass[2010]{Primary 34L05, 34B07; Secondary 34B20, 37K15}

\begin{abstract}
 We discuss direct and inverse spectral theory for a Sturm--Liouville type problem with a quadratic dependence on the eigenvalue parameter, which arises as the isospectral problem for the conservative Camassa--Holm flow.  
\end{abstract}

\maketitle

\section{Introduction}

 The principal purpose of the present article is to discuss direct and inverse spectral theory for a Sturm--Liouville type problem of the form    
 \begin{align}\label{eqnISP}
  -f'' + \frac{1}{4} f = z\, \omega f + z^2 \dip f, 
 \end{align} 
 where $\omega$ is a real-valued Borel measure on $\R$, $\dip$ is a non-negative Borel measure on $\R$ and $z$ is a complex spectral parameter. 
The significance of this rather specific spectral problem stems from the fact that it arises as the isospectral problem of a particular completely integrable nonlinear wave equation.
 More precisely, it has been identified as an isospectral problem for the two-component Camassa--Holm system \cite{clz06, hoiv12} and it turned out recently \cite{ConservMP} that it also serves as an isospectral problem for global conservative solutions of the Camassa--Holm equation \cite{brco07, grhora12, hora07}.
 Regarding further information about the Camassa--Holm equation, we only refer to a brief selection of articles \cite{bkst09, caho93, co01, coes98, como00, cost00, mc03, mc04}.  

 Inverse spectral and scattering theory for the Sturm--Liouville type problem~\eqref{eqnISP} is of peculiar interest for solving the Cauchy problem for the Camassa--Holm equation and it's two-component generalization. 
 Since the coefficient $\omega$ is allowed to change sign and because of the presence of the measure $\dip$, spectral theory for~\eqref{eqnISP} is outside of most standard theory for Sturm--Liouville problems and requires distinct methods to deal with it. 
 In particular, direct and inverse spectral theory for~\eqref{eqnISP} is still not sufficiently developed for applications to the Camassa--Holm flow (but see \cite{besasz00, be04, bebrwe08, bebrwe12, co01, cogeiv06, LeftDefiniteSL, ConservMP, IsospecCH, gewe14}). 
  Moreover, except for \cite{ConservMP}, all of these references only deal with the case when the measure $\dip$ is not present at all. 
 However, let us also mention that problems similar to~\eqref{eqnISP}  have been studied in \cite{krla79, krla80, la76, lawi98} in the context of indefinite strings, where the authors dealt with the spectral problem in a Krein space setting. 
 
 In this article, we provide a thorough operator theoretic framework to treat the spectral problem~\eqref{eqnISP} which will serve as a solid basis for further investigations on the integrability of the conservative Camassa--Holm flow. 
 More precisely, we will provide basic self-adjointness results for realizations of this spectral problem (on an interval $J\subseteq\R$) in Hilbert spaces of the form 
\begin{align} 
 \cH(J) = H^1(J) \times L^2(J;\dip),
\end{align}
 equipped with a suitable scalar product. 
 These self-adjoint realizations are mostly of an auxiliary nature, whereas the more convenient objects seem to be associated quadratic operator pencils in $H^1(J)$ which will be introduced next.
 We will also introduce (singular) Weyl--Titchmarsh functions, which are basic objects of spectral theory for Sturm--Liouville problems (for further information on singular Weyl--Titchmarsh functions we refer to \cite{fu08, fulalu12, gezi06, ka67, ko49, kosate12}).
 All this will be done for the cases of bounded intervals (in Section \ref{sec:03}), semi-axes (in Section \ref{sec:04}) and the whole line (in Section \ref{secWL}) separately. 
 Even though it could be done at once in principle, we decided to present all these cases separately for the sake of clearness and to avoid distracting case differentiations and awkward notation.
 Since the whole line case is of particular importance for applications to the Camassa--Holm flow, we will furthermore introduce a spectral measure and a spectral transformation in this case as well. 
 In the final Section~\ref{secIST}, we will provide several basic inverse uniqueness theorems for the spectral problem~\eqref{eqnISP} following~\cite{LeftDefiniteSL}. 
 More precisely, we will provide Borg--Marchenko type uniqueness results for the spectral problem on semi-axes as well as some uniqueness results for the whole line.  

 Although our main motivation lies in applications to the conservative Camassa--Holm flow, we think that the present article is also of interest to a wider audience since it provides a new way to treat Sturm--Liouville type problems with a quadratic dependence on the spectral parameter. 
 The theory developed in this article for example also works for more general problems of the form  
 \begin{align}
  - f'' + \chi f = z\, \omega f + z^2 \dip f, 
 \end{align}
 where $\chi$ is a non-negative Borel measure on $\R$.  
 In this context, let us also mention that the spectral problem~\eqref{eqnISP} can be transformed via a Liouville transform to a Schr\"odinger spectral problem with an energy dependent potential 
 \begin{align}\label{eqnSPS}
  - f'' + q f + z\, p f = z^2 f, 
 \end{align}
 provided that the measures $\omega$ and $\dip$ are sufficiently smooth and positive. 
 Spectral problems \eqref{eqnSPS} arise in various contexts and we just mention \cite{hrpr12, jaje72, ka75, sasz96} for further information and references. 
 However, for our applications this transformation is not possible since we have to allow more general coefficients $\omega$ and $\dip$. 
  
\subsection*{Notation} 

 For every open interval $J\subseteq \R$, we denote with $H^1(J)$ and $H^1_0(J)$ the usual Sobolev spaces equipped with the modified scalar product
\be\label{eq:modH1}
 \spr{f}{g}_{H^1(J)} = \frac{1}{4}\int_{J}f(x)g(x)^\ast dx + \int_{J}f'(x)g'(x)^\ast dx, \quad f,\,g\in H^1(J). 
\ee
 With $H_{\mathrm{c}}^1(J)$, we denote the dense subspace of functions in $H^1_0(J)$ which have compact support in $J$. 
 If $J\subseteq\R$ is a general (not necessarily open) interval, then the spaces $H^1(J)$, $H^1_0(J)$ and $H^1_{\mathrm{c}}(J)$ simply denote the respective spaces corresponding to the interior of $J$.  
 Moreover,  we will need the space 
 \begin{align}
H_{\mathrm{loc}}^1(\R) = \{f\in AC_{\mathrm{loc}}(\R) \,|\, fg\in H^1_{\mathrm{c}}(\R)\ \text{for all}\ g\in H^1_{\mathrm{c}}(\R)\}.
 \end{align}

As we are dealing with measure coefficients, we employ the following convenient notation: 
For integrals with respect to some Borel measure $\mu$ on $\R$ we set
\begin{align}
 \int_x^y f d\mu = \begin{cases}
                                     \int_{[x,y)} f d\mu, & y>x, \\
                                     0,                                     & y=x, \\
                                     -\int_{[y,x)} f d\mu, & y< x, 
                                    \end{cases}
\end{align}
rendering the integral left-continuous as a function of $y$. 
 Furthermore, we will make extensive use of the following integration by parts formula for Borel measures $\mu$, $\nu$ on $\R$  (see, for example, \cite[Theorem~21.67]{hest65}):
\begin{align}\label{eqnPI}
\int_{x}^y F(s)d\nu(s) = \left. FG\right|_x^y - \int_{x}^y  G(s+)d\mu(s), \quad x,\,y\in\R,
\end{align}
where $F$, $G$ are left-continuous distribution functions of $\mu$, $\nu$, respectively.

\section{The basic differential equation}\label{secBDE}

Throughout this article, we let $\omega$ be a real-valued Borel measure on $\R$ and $\dip$ be a non-negative Borel measure on $\R$. 
As already mentioned in the introduction, the main object of interest is the inhomogeneous ordinary differential equation 
 \begin{align}\label{eqnDEinho}
  -f'' + \frac{1}{4} f = z\, \omega f + z^2 \dip f + \chi, 
 \end{align}
 where $\chi$ is a complex-valued Borel measure on $\R$ and $z\in\C$ is a complex spectral parameter. 
 Of course, this equation has to be understood in a distributional sense, where the right-hand side is a Borel measure as soon as the function $f$ is at least continuous. 
 To be precise, a solution of the differential equation~\eqref{eqnDEinho} is a locally absolutely continuous function $f$ on $\R$ such that 
 \begin{align}\label{eqnDEint}
  - f'(y) + f'(x) + \frac{1}{4} \int_{x}^{y} f(s)ds = z \int_{x}^{y} f d\omega + z^2 \int_{x}^y f d\dip + \int_{x}^{y} d\chi 
 \end{align}
 for some $y\in\R$ and almost all $x\in\R$. 
 In particular, the derivative of such a solution $f$ has a representative which is locally of bounded variation, such that the limits 
 \begin{align}
  f'(x\pm) = \lim_{\varepsilon\downarrow0} f'(x\pm\varepsilon)
 \end{align}
 exist for all $x\in\R$ and coincide, except possibly for the points where one of our measures has mass. 
 For definiteness, we will always choose the unique left-continuous representative for $f'$ such that~\eqref{eqnDEint} holds for all $x$, $y\in\R$.  

\subsection{Existence and uniqueness}

  As a preliminary step, we will first derive a few basic results about the differential equation~\eqref{eqnDEinho} from the general theory of measure differential equations. 
  In particular, this connection immediately yields the following existence and uniqueness result for our differential equation. 

\begin{lemma}\label{lemEE}
 For every complex-valued Borel measure $\chi$ on $\R$, $c\in\R$ and $d_1$, $d_2$, $z\in\C$ there is a unique solution $f$ of the differential equation~\eqref{eqnDEinho} with 
 \begin{align}\label{eqnDEiv}
 f(c) & = d_1, & f'(c) & = d_2.
 \end{align} 
 If $\chi$ is real-valued as well as $d_1$, $d_2$, $z\in\R$, then the solution $f$ is real-valued too. 
\end{lemma}

\begin{proof}
 Existence and uniqueness follows readily from \cite[Corollary 2.2]{pe88} (see also \cite[Section~11.8]{at64}, \cite[Theorem~1.1]{be89},  \cite[Theorem~A.2]{MeasureSL}). 
 An inspection of the fixed-point iteration in \cite{pe88} shows that real-valued data yields real-valued solutions.  
\end{proof}

Of course, we could also prescribe the right-hand limit of the derivative of our solution (instead of the left-hand limit) in~\eqref{eqnDEiv} and still obtain a unique solution. 

In order to provide a representation of solutions to the inhomogeneous differential equation~\eqref{eqnDEinho}, we also consider the  corresponding homogeneous equation, 
 \begin{align}\label{eqnDEho}
  -f'' + \frac{1}{4} f = z\, \omega f + z^2 \dip f.
\end{align}
Using integration by parts, it is readily verified that the usual Wronski determinant 
\begin{align}
W(\theta,\phi)(x) = \theta(x)\phi'(x)-\theta'(x)\phi(x), \quad x\in\R,
\end{align}
of two solutions $\theta$, $\phi$ to the homogeneous differential equation~\eqref{eqnDEho} is constant. 
Indeed, set $F=\phi'$ and $G=\theta$ in \eqref{eqnPI} first and then $F=\theta'$ and $G=\phi$. 
Subtracting one of these relations from the other one shows that $W(\theta,\phi)(x)$ does not depend on $x\in\R$.
Moreover, this constant is non-zero if and only if the solutions $\theta$ and $\phi$ are linearly independent. 
In this case, the pair of solutions $\theta$, $\phi$ is called {\em a fundamental system} of the homogeneous differential equation~\eqref{eqnDEho} if furthermore $W(\theta,\phi)=1$.  
Note that, due to Lemma \ref{lemEE}, such fundamental systems always exist. 

\begin{corollary}\label{cor:inhom}
 Let $\chi$ be a  complex-valued Borel measure on $\R$, $c\in\R$ and $z\in\C$. 
 If $\theta$, $\phi$ is a fundamental system of the homogeneous differential equation~\eqref{eqnDEho}, then any solution $f$ of the differential equation~\eqref{eqnDEinho} can be written as 
\begin{align}\label{eq:inhom}
 f(x) = d_1 \theta(x) + d_2 \phi(x)+\int_c^x{\theta(x)\phi(s)-\theta(s)\phi(x)}d\chi(s), \quad x\in\R,
\end{align}
for some constants $d_1$, $d_2\in\C$. 
\end{corollary}

\begin{proof}
 Using the integration by parts formula \eqref{eqnPI}, 
  one verifies that the derivative of the function on the right-hand side of~\eqref{eq:inhom} is given by 
 \begin{align*}
  d_1 \theta'(x) + d_2 \phi'(x)+\int_c^x{\theta'(x)\phi(s)-\theta(s)\phi'(x)}d\chi(s), \quad x\in\R.
\end{align*}
 Upon integrating by parts once more, one shows that this function is a solution of the differential equation~\eqref{eqnDEinho} indeed. 
 Now after choosing
 \begin{align*}
  \begin{pmatrix} d_1 \\ d_2 \end{pmatrix} = \begin{pmatrix} \theta(c) & \phi(c) \\ \theta'(c) & \phi'(c) \end{pmatrix}^{-1} \begin{pmatrix} f(c) \\ f'(c) \end{pmatrix} = \begin{pmatrix} \phi'(c) & - \phi(c) \\ -\theta'(c) & \theta(c) \end{pmatrix} \begin{pmatrix} f(c) \\ f'(c) \end{pmatrix}, 
 \end{align*}
 the claim follows from the uniqueness part of Lemma~\ref{lemEE}.
\end{proof}

As a final result of this subsection, we show that the solutions of the differential equation~\eqref{eqnDEinho} with fixed initial conditions of the form~\eqref{eqnDEiv} depend analytically on the complex spectral parameter $z\in\C$. 

\begin{lemma}\label{lemSolEnt}
  Let $\chi$ be a  complex-valued Borel measure on $\R$, $c\in\R$ and $d_1$, $d_2\in\C$. 
  If for every $z\in\C$, the unique solution of the differential equation~\eqref{eqnDEinho} with the initial conditions~\eqref{eqnDEiv} is denoted by $f_z$, then the functions 
 \begin{align}\label{eq:fz}
  z & \mapsto f_z(x), & z & \mapsto f_z'(x),
  \end{align}
  are entire for every $x\in\R$.  
\end{lemma}

\begin{proof}
 The claim follows from the fixed-point iteration which converges locally uniformly in $z\in\C$ (see \cite[\S 1.3]{la76}, \cite[Lemma~1.5]{be89}, \cite[Theorem~A.5]{MeasureSL}). 
\end{proof}

Let us mention that the entire functions in~\eqref{eq:fz} are of finite exponential type for every $x\in\R$ (cf.\ \cite[\S 1.3]{la76} and \cite[\S 2]{kk74}).  
We will derive this fact for the homogeneous differential equation~\eqref{eqnDEho} in an effortless manner in Corollary~\ref{corFScartwright} below.

\subsection{An associated linear relation}

The differential equation~\eqref{eqnDEinho} gives rise to a linear relation $\Tloc$ in the space of functions $\Hloc = H_{\loc}^{1}(\R) \times L_{\loc}^2(\R;\dip)$.
More precisely, this linear relation is defined by saying that some pair $(f,g)\in\Hloc\times\Hloc$ belongs to $\Tloc$ if and only if  
\begin{align}\label{eqnDErel}
 -f_1'' + \frac{1}{4} f_1 & = \omega g_{1} + \dip g_{2}, &       \dip f_2 & = \dip g_{1}. 
\end{align}
Here, the subscripts denote the respective component of a pair in $\Hloc$. 
Furthermore, we will also employ the following convenient notation for elements of the linear relation $\Tloc$: 
Given some $\f\in \Tloc$, we will denote its first component with $f$ and its second one with $\tau f$ (although $\Tloc$ is in general not an operator). 

The linear relation $\Tloc$ is closely related to the differential equation~\eqref{eqnDEinho}. 
For any $z\in\C$, a pair $(f,g)\in\Hloc\times\Hloc$ belongs to $\Tloc-z$ if and only if 
\begin{align}\label{eqnTloc-z}
 -f_1'' + \frac{1}{4}f_1 & = z\, \omega f_1 +z^2 \dip f_1 + \omega g_1 + z\,\dip g_1 + \dip g_2, & \dip f_2 = \dip g_1 + z\,\dip f_1. 
\end{align}
This simple observation immediately shows that 
\begin{align}\label{eq:Tlocrake}
 \ran{\Tloc - z} & = \Hloc, & \dim\ker{\Tloc - z} & = 2,
\end{align}
in view of the existence and uniqueness result in Lemma~\ref{lemEE}. 
Moreover, some $f$ belongs to $\ker{\Tloc-z}$ if and only if $f_1$ is a solution of the homogeneous differential equation~\eqref{eqnDEho} and $\dip f_2 = z\, \dip f_1$. 
This shows that the kernel of $\Tloc-z$ can be identified with the space of solutions of the homogeneous differential equation~\eqref{eqnDEho}. 

As already mentioned above, the linear relation is in general not an operator. 
In fact, the very definition shows that the multi-valued part of $\Tloc$ is given by
\be\label{eq:multpart}
\mul{\Tloc}=\left\{ h\in\Hloc \,| \ h_2=0,\  \omega h_1 = \dip h_1=0 \right\}. 
\ee
This is verified by noting that some $h\in\Hloc$ belongs to the multi-valued part of $\Tloc$ if and only if $\dip h_1=0$ and $\omega h_1+\dip h_2=0$. 
Now the representation in~\eqref{eq:multpart} of the multi-valued part of $\Tloc$ is readily deduced from this equivalence. 

One of our main interests in the following sections lies in the investigation of self-adjointness of the linear relation $\Tloc$ when restricted to suitable Hilbert spaces. 
In this respect, an important role will be taken by the modified Wronskian 
\begin{align}\label{eqWrmod}
 V(\f,\g)(x) = \tau f_1(x) g_1'(x) - f_1'(x) \tau g_1(x), \quad x\in\R,
\end{align}
defined for every $\f$, $\g\in \Tloc$. 
We note that the function $V(\f,\g)$ is left-continuous and locally of bounded variation as the following Lagrange identity shows.

\begin{proposition}\label{propLagrange}
 For every $\f$, $\g\in \Tloc$ and $x$, $y\in\R$ we have 
 \begin{align}
 \begin{split}\label{eq:lag}
  V(\f,\g)(y) - V(\f,\g)(x) =  \frac{1}{4} \int_{x}^{y} & \tau f_1(s)g_1(s) - f_1(s) \tau g_1(s) ds\\  
  &  +  \int_{x}^{y} \tau f_1'(s)g_1'(s) - f_1'(s) \tau g_1'(s) ds  \\ 
                                       &  +  \int_{x}^{y} \tau f_{2}(s) g_2(s) - f_2(s) \tau g_{2}(s) d\dip(s).
 \end{split}
 \end{align}
\end{proposition}

\begin{proof}
 An integration by parts, using the definition of $\Tloc$ in~\eqref{eqnDErel}, shows that
 \begin{align}\begin{split}\label{eqnInP}
  \frac{1}{4} \int_x^y & f_1(s) h_1(s) ds + \int_x^y f_1'(s) h_1'(s) ds \\ 
    & = \int_x^y \tau f_1(s) h_1(s) d\omega(s) + \int_x^y \tau f_2(s) h_1(s) d\dip(s) + [f_1' h_1]_{x}^y
 \end{split}\end{align}
 holds for every $h\in\Hloc$.
 In particular, choosing $h=\tau g$ and subtracting the corresponding equation with the roles of $\f$ and $\g$ reversed (also taking into account the second equation in~\eqref{eqnDErel}) yields the claim.
\end{proof}

Given some $z\in\C$ and $f$, $g\in\ker{\Tloc-z}$, the corresponding pairs $\f = (f,zf)$ and $\g=(g,zg)$ clearly belong to $\Tloc$. 
From the Lagrange identity, it is now readily seen that the modified Wronskian $V(\f,\g)$ is constant on $\R$. 
In fact, this also follows from the connection with the usual Wronski determinant,  
\begin{align}\label{eqnVtoW}
V(\f,\g)(x) = z f_1(x) g_1'(x) -  f_1'(x)\, z g_1(x) = z\, W(f_1,g_1)(x), \quad x\in\R,
\end{align}
which holds in this case. 
In particular, unless $z$ equals zero, this also guarantees that the modified Wronskian $V(\f,\g)$ is non-zero if and only if $\f$ and $\g$ are linearly independent. 
This one exception already foreshadows the somewhat distinct role of the case when $z$ is zero, as then one always has $V(\f,\g)=0$.

\section{The spectral problem on a bounded interval}\label{sec:03}

In this section, we will first discuss some spectral theory for the differential equation~\eqref{eqnDEinho} on a bounded interval of the form $[a,b)$ for fixed $a$, $b\in\R$ with $a<b$. 
As a suitable setting for this purpose, we consider the Hilbert space  
\begin{align}
 \Hab = H^1([a,b))\times L^2([a,b);\dip),
\end{align}
 equipped with the scalar product
\begin{align}\begin{split}
 \spr{f}{g}_{\Hab} = \frac{1}{4} \int_{a}^b  f_1(x) g_1(x)^\ast & dx  + \int_{a}^b f_1'(x) g_1'(x)^\ast dx \\
                            &  + \int_{a}^b f_2(x) g_2(x)^\ast d\dip(x), \quad f,\,g\in\Hab. 
\end{split}\end{align} 
Apart from this, we also introduce the closed linear subspace   
\begin{align} 
 \Hoab = H_0^1([a,b)) \times L^2([a,b);\dip).
\end{align}
 Clearly, point evaluations of the first component are continuous on $\cH([a,b))$. 

\subsection{Self-adjointess of the spectral problem on a bounded interval}\label{ssec31}

 As a first step, we introduce the maximal relation $\Tmax$ in $\Hab$ by restricting $\Tloc$; 
 \begin{equation}\label{eq:Tmaxab}
  \Tmax = \left\lbrace \f\in\Hab\times\Hab \left|\, \f=\g|_{[a,b)} ~\text{for some }\g\in\Tloc \right.\right\rbrace.
 \end{equation}
 Given $\f\in\Tmax$, we will usually identify it with any representative in $\Tloc$. 
 In this respect, one notes that the quantities $f_1(x)$, $f_1'(x)$ as well as $\tau f_1(x)$ are independent of the actually chosen representative in $\Tloc$ for all $x$ in the closed interval $[a,b]$. 
 As a consequence, the modified Wronskian $V(\f,\g)$ is also well-defined on $[a,b]$ for every $\f$, $\g\in\Tmax$, and the Lagrange identity~\eqref{eq:lag} holds for all $x$, $y\in[a,b]$.  

 It is an immediate consequence of~\eqref{eq:Tlocrake} and the very definition of $\Tmax$ that   
 \begin{align}\label{eq:ran_ab}
  \ran{\Tmax} & = \Hab = \Hoab \oplus \ker{\Tmax}, 
  \end{align}
 where the orthogonality follows from a simple integration by parts as in~\eqref{eqnInP}. 

Since $\Tmax$ turns out not to be self-adjoint, we also consider the restriction 
\begin{equation}\label{eq:Tminab}
\Tmin = \left\lbrace \f\in \Tmax \,|\, \tau f_1(a) = f_1'(a) = \tau f_1(b) = f_1'(b) = 0 \right\rbrace,
\end{equation}
which is referred to as the minimal relation. 
 Here, we have   
\begin{align}\label{eq:ran_oab}
  \ran{\Tmin} & =\Hoab, 
\end{align}
as for every $g\in \Hloc$ there is a solution $f\in\Hloc$ of the differential equation~\eqref{eqnDErel} with $f_1'(a) = f_1'(b) = 0$. 
In fact, this can always be achieved by suitably adjusting the constants $d_1$, $d_2\in\C$ in Corollary~\ref{cor:inhom}. 

\begin{theorem}\label{th:3.01}
 The minimal relation $\Tmin$ is symmetric in $\Hab$ with 
 \begin{align}\label{eq:T=T*ab}
   \Tmin^\ast & = \Tmax, & \Tmax^\ast & = \Tmin. 
 \end{align}
 In particular, the linear relations $\Tmin$ and $\Tmax$ are closed. 
\end{theorem}

\begin{proof}
From the Lagrange identity~\eqref{eq:lag} it is immediate that
\begin{align}\label{eq:lag2}
\spr{\tau f}{g}_{\Hab} - \spr{f}{\tau g}_{\Hab} =V(\f,\g^\ast)(b) - V(\f,\g^\ast)(a)
\end{align}
holds for all $\f$, $\g\in \Tmax$. 
In particular, this implies that 
\begin{align*}
\spr{f}{\tau g}_{\Hab} = \spr{\tau f}{g}_{\Hab},
\end{align*}
whenever $\f\in \Tmin$ and $\g\in\Tmax$, which guarantees 
\begin{align*}
 \Tmax & \subseteq\Tmin^\ast, & \Tmin & \subseteq \Tmax^*.
\end{align*}
In order to show that $\Tmin^\ast\subseteq \Tmax$, fix some $(f,f_\tau)\in \Tmin^\ast$. 
Because of~\eqref{eq:ran_ab}, there is an $h\in \dom{\Tmax}$ such that $(h,f_\tau)\in\Tmax$ and thus 
\begin{align*}
\spr{h}{\tau g}_{\Hab}= \spr{f_\tau}{g}_{\Hab}, \quad \g\in \Tmin.
\end{align*} 
On the other hand, since $(f,f_\tau)\in \Tmin^\ast$, we also conclude that  
\begin{align*}
\spr{f}{\tau g}_{\Hab}= \spr{f_\tau}{g}_{\Hab}, \quad\g\in \Tmin.
\end{align*}
As a consequence of these equations, we infer that $h-f\in\ker{\Tmax}$ in view of~\eqref{eq:ran_ab} and~\eqref{eq:ran_oab}.
 Thus we conclude $(f,f_\tau)\in\Tmax$, which proves the first equality in \eqref{eq:T=T*ab}.
 In particular, $\Tmax$ turns out to be closed and the second equality in~\eqref{eq:T=T*ab} follows upon noting that $\Tmin$ is closed as well, since the linear functionals $\f\mapsto\tau f_1(a)$, $\f\mapsto f_1'(a)$, $\f\mapsto\tau f_1(b)$ and $\f\mapsto f_1'(b)$ are bounded on $\Tmax$ (cf.\ \cite[Lemma~2.6]{LeftDefiniteSL}). 
 \end{proof}
 
 Given some $z\in\C$, we note that the kernel of $\Tmax-z$ can be identified with the space of solutions of the homogeneous differential equation~\eqref{eqnDEho}.
 In particular, this shows that the dimension of the kernel is two, that is,  
 \begin{align}\label{eqnKerTmax}
  \dim\ker{\Tmax-z} = 2, 
 \end{align}
 which immediately yields the following result. 
 
 \begin{corollary}
 The minimal relation $\Tmin$ has  deficiency indices equal to two. 
 \end{corollary}

 As a consequence, this guarantees that the minimal relation $\Tmin$ has self-adjoint extensions.  
 Using the Lagrange identity \eqref{eq:lag2}, one can obtain a complete description of all of them in a standard way.  
 Among these self-adjoint extensions, we shall distinguish the following two kinds:
 \begin{enumerate}[label=(\roman*), ref=(\roman*), leftmargin=*, widest=iii]  
  \item {\it Separated boundary conditions}: For all $\alpha$, $\beta\in[0,\pi)$, we set
   \begin{align}\label{eq:bcsep}
    \T_{\alpha,\beta} = \left\lbrace \f\in \Tmax \,\left| \begin{array}{r} \tau f_1(a)\cos\alpha - f_1'(a) \sin\alpha  = 0\\
         \tau f_1(b) \cos\beta - f_1'(b) \sin\beta = 0 \end{array}\right.\right\rbrace.
   \end{align}
  \item {\it $\vartheta$-periodic boundary conditions}: For all $\vartheta\in [0,2\pi)$, we set
   \begin{align}\label{eq:bcper}
    \T_{\vartheta} = \left\lbrace \f\in \Tmax \,\left|\,  \begin{pmatrix} f_1'(a) \\  \tau f_1(a) \end{pmatrix} = \E^{\I \vartheta} \begin{pmatrix} f_1'(b) \\ \tau f_1(b) \end{pmatrix} \right.\right\rbrace.
   \end{align}
 \end{enumerate}
 In this article, we will only discuss self-adjoint extensions $\T_{\alpha,\beta}$ of $\Tmin$ with separated boundary conditions.  
 The eigenvalues of this linear relation are precisely those $z\in\C$ for which there is a non-trivial solution $\phi$ of the homogeneous differential equation~\eqref{eqnDEho} which satisfies the boundary conditions 
\begin{align}
 \label{eqnBCa}  z f(a) \cos\alpha - f'(a)\sin\alpha & =0, \\ 
 \label{eqnBCb} z f(b)\cos\beta - f'(b)\sin\beta & =0,
\end{align}
at $a$ and $b$. 
Before we prove that the spectrum of the self-adjoint linear relation $\T_{\alpha,\beta}$ is purely discrete, we first derive a representation for the resolvent. 

\begin{theorem}\label{th:Resab}
 The spectrum of the self-adjoint linear relation $\T_{\alpha,\beta}$ is purely discrete. 
 If some nonzero $z\in\C$ belongs to the resolvent set of $\T_{\alpha,\beta}$, then   
  \begin{align}\label{eqnResab}
     z\, (\T_{\alpha,\beta}-z)^{-1} g(x) & = \spr{g}{\G(x,\cdot\,)^\ast}_{\Hab} \begin{pmatrix} 1 \\ z \end{pmatrix} - g_1(x) \begin{pmatrix} 1 \\ 0 \end{pmatrix}, \quad x\in[a,b),  
  \end{align}
 for every $g\in\Hab$, where the Green's function $\G$ is given by  
 \begin{align}
  \G(x,s) = \begin{pmatrix} 1 \\ z \end{pmatrix} \frac{1}{W(\psi,\phi)} \begin{cases} \psi(x) \phi(s), & s\leq x, \\ \psi(s) \phi(x), & s> x,    \end{cases}
 \end{align}
 and $\psi$, $\phi$ are linearly independent solutions of the homogeneous differential equation~\eqref{eqnDEho} such that $\psi$ satisfies the boundary condition~\eqref{eqnBCb} at $b$ and $\phi$ satisfies the boundary condition~\eqref{eqnBCa} at $a$. 
\end{theorem}

\begin{proof}
 First of all, pick some nonzero $z\in\C$ and assume that it is not an eigenvalue of $\T_{\alpha,\beta}$. 
 Note that in this case, non-trivial solutions $\psi$, $\phi$ of the homogeneous differential equation~\eqref{eqnDEho} with the required boundary conditions always exist. 
 Moreover, they are linearly independent since otherwise $z$ would be an eigenvalue of $\T_{\alpha,\beta}$. 
 If some pair $(f,g)$ belongs to $\T_{\alpha,\beta} - z$, then~\eqref{eqnTloc-z} and integrating by parts shows  
\begin{align*}
 \int_a^b & \G_1(x,s)g_1(s)d\omega(s) + \int_a^b \G_1(x,s) (z\, g_1(s) + g_2(s))d\dip(s) \\ 
               & = f_1(x) -  \frac{\phi(x)}{W(\psi,\phi)} W(\psi,f_1)(b) -  \frac{\psi(x)}{W(\psi,\phi)} W(f_1,\phi)(a), \quad x\in[a,b). 
\end{align*}
On the other side, after another integration by parts, one also has for $x\in[a,b)$ 
\begin{align*}
 z \int_a^b & \G_1(x,s)g_1(s)d\omega(s) + z^2 \int_a^b \G_1(x,s)g_1(s)d\dip(s)  \\ 
                  & = \spr{g_1}{\G_1(x,\redot)^\ast}_{H^1([a,b))} - \frac{\phi(x)}{W(\psi,\phi)}g_1(b)\psi'(b) + \frac{\psi(x)}{W(\psi,\phi)} g_1(a)\phi'(a) - g_1(x). 
\end{align*}
After comparing these equations, and taking into account the boundary conditions
\begin{align*}
(g_1(a)+zf_1(a))\cos\alpha-f_1'(a)\sin\alpha & =0,\\
(g_1(b)+zf_1(b))\cos\beta-f_1'(b)\sin\beta & =0,
\end{align*}
we finally obtain that $f_1$ is given as in~\eqref{eqnResab}, since $z$ is non-zero. 
Furthermore, as an immediate consequence of the definition of $\T_{\alpha,\beta}$ we have 
\begin{align*}
 f_2(x)  = g_1(x)+ z f_1(x)= \spr{g}{\G(x,\redot)^\ast}_{\Hab},
\end{align*}
for almost all $x\in[a,b)$ with respect to $\dip$, which proves that the inverse of $\T_{\alpha,\beta}-z$ is given as in~\eqref{eqnResab} for every $g\in\ran{\T_{\alpha,\beta}-z}$.  
However, an inspection of this expression implies that the inverse of $\T_{\alpha,\beta}-z$ is bounded and hence $z$ actually belongs to the resolvent set of $\T_{\alpha,\beta}$.
This already proves that the nonzero spectrum of $\T_{\alpha,\beta}$ consists only of eigenvalues as well as the representation of the resolvent. 

Next one notes that some $z\in\C$ is an eigenvalue of $\T_{\alpha,\beta}$ if and only if 
\begin{align}\label{eqnSARboundSDS}
 z\phi_\alpha(z,b) \cos\beta - \phi_\alpha'(z,b) \sin\beta = 0,  
\end{align}
 where $\phi_\alpha(z,\redot)$ denotes the solution of~\eqref{eqnDEho} with the initial conditions
\begin{align*}
  \phi_\alpha(z,a) & = \sin\alpha, & \phi_\alpha'(z,a) & = z \cos\alpha. 
\end{align*}
More precisely, this follows (for nonzero $z$) since every solution of~\eqref{eqnDEho} which satisfies the boundary condition~\eqref{eqnBCa} at $a$ is a scalar multiple of $\phi_\alpha(z,\redot)$. 
The case when $z$ is zero can be verified explicitly. 
In view of Lemma~\ref{lemSolEnt}, the function on the left-hand side of~\eqref{eqnSARboundSDS} is entire in $z$. 
Now if zero belongs to the spectrum $\T_{\alpha,\beta}$, then this means that it is an isolated point of the spectrum and hence an eigenvalue of $\T_{\alpha,\beta}$. 
This shows that the spectrum of $\T_{\alpha,\beta}$ consists only of eigenvalues with no finite accumulation point. 
Since the multiplicity of every eigenvalue is at most two by~\eqref{eqnKerTmax}, this shows that the spectrum of $\T_{\alpha,\beta}$ is purely discrete. 
\end{proof}

\begin{corollary}
 Every nonzero eigenvalue of $\T_{\alpha,\beta}$ is simple. 
\end{corollary}

\begin{proof}
 If $z\in\C$ and $f$, $g\in\ker{\T_{\alpha,\beta}-z}$, then
\begin{align*}
 z\, W(f_1,g_1) = zf_1(a) g_1'(a) - f_1'(a)\, z g_1(a) = 0,
\end{align*}
which shows that $f$, $g$ are linearly dependent as long as $z$ is non-zero. 
\end{proof}

The somewhat distinct role of the case when $z$ is zero is due to the appearance of the factor $z$ in the boundary conditions in~\eqref{eqnBCa} and~\eqref{eqnBCb}. 
In this case, these boundary conditions either reduce to Neumann boundary conditions or become void at all. 
Since the kernel of $\T_{\alpha,\beta}$ can be given explicitly, we obtain  
\begin{align}\label{eq:kernel}
 \dim\ker{\T_{\alpha,\beta}} = \begin{cases}  
 0, & \text{if neither } \alpha \text{ nor } \beta \text{ is zero}, \\ 
 1, & \text{if either } \alpha \text{ or } \beta \text{ is zero}, \\ 
 2, & \text{if both, } \alpha \text{ and } \beta \text{ are zero}.  \end{cases}
\end{align}

\subsection{A quadratic operator pencil on a bounded interval}

In this subsection, we will introduce a quadratic operator pencil in $H_0^1([a,b))$, associated with the differential equation~\eqref{eqnDEinho} and Dirichlet boundary conditions at the endpoints. 
To this end, we first consider the linear relation $\T_0$ in $\Hoab$, defined by  
 \begin{align}
  \T_0 = \left\lbrace \f\in\T_{0,0} \,|\, \f\in\Hoab\times\Hoab \right\rbrace.
 \end{align}
 In view of~\eqref{eq:ran_ab}, one immediately sees that  
 \begin{align}
  \T_{0,0} = \T_0 \oplus \left(\ker{\T_{0,0}}\times\lbrace0\rbrace\right),
 \end{align}
 which already guarantees that the linear relation $\T_0$ is self-adjoint in the Hilbert space $\Hoab$ and has purely discrete spectrum with $\sigma(\T_{0,0}) = \sigma(\T_{0})\cup\lbrace0\rbrace$. 

 \begin{theorem}\label{thm:T0inv}
  Zero belongs to the resolvent set of $\T_0$ with 
  \begin{align}\label{eqnT0invBlock}
   \T_{0}^{-1} = \begin{pmatrix} \Omega_0 & \Upsilon_0  \\ \mathrm{I}_0 & 0 \end{pmatrix},  
  \end{align}
 where the operator $\Omega_0: H_0^1([a,b))\rightarrow H_0^1([a,b))$ is given by 
\begin{align}
 \Omega_0 g_1(x) =  \int_{a}^b K_0(x,s) g_1(s)d\omega(s), \quad x\in [a,b),~ g_1\in H_0^1([a,b)),
\end{align}
the operator $\Upsilon_0: L^2([a,b);\dip) \rightarrow H_0^1([a,b))$ is given by 
\begin{align}
 \Upsilon_0 g_2(x) = \int_{a}^b K_0(x,s) g_2(s)d\dip(s), \quad x\in [a,b),~ g_2\in L^2([a,b);\dip),
\end{align}
 and the operator $\mathrm{I}_0: H_0^1([a,b)) \rightarrow L^2([a,b);\dip)$ is the canonical embedding. Here,  
  \begin{align}
   K_0(x,s) =  \frac{2}{\sinh\left(\frac{b-a}{2}\right)} \begin{cases} \sinh\left(\frac{b-x}{2}\right) \sinh\left(\frac{s-a}{2}\right), & s\leq x,   \\
                                                                    \sinh\left(\frac{b-s}{2}\right) \sinh\left(\frac{x-a}{2}\right), & s>x.  \end{cases}
 \end{align} 
 \end{theorem}

 \begin{proof} 
  The fact that zero belongs to the resolvent set of $\T_0$ follows from the very definition of $\T_0$ and Theorem~\ref{th:Resab}.  
  Similarly to the proof of Theorem~\ref{th:Resab}, using~\eqref{eqnDErel}, integration by parts and the boundary conditions, one obtains for every $(f,g)\in\T_0$   
  \begin{align*}
   f_1(x) = \int_{a}^b  K_0(x,s) g_1(s) d\omega(s)  + \int_{a}^b K_0(x,s) g_2(s) d\dip(s), \quad x\in[a,b).
  \end{align*}
  Furthermore, by the definition of $\T_0$, one has $f_2(x) = g_1(x)$ for almost all $x\in[a,b)$ with respect to $\dip$.  
  In particular, all this ensures that the operators $\Omega_0$, $\Upsilon_0$ and $\mathrm{I}_0$ in the claim are bounded (as the inverse of $\T_0$ is as well). 
 \end{proof}
 
 We now define the quadratic operator pencil $\Pe_0$ in $H^1_0([a,b))$ by setting 
\begin{align}
 \Pe_0(z) = \mathrm{I}_0 - z\, \Omega_0 - z^2 \Upsilon_0, \quad z\in\C.
\end{align}
Here, by abuse of notation, we reuse the symbol $\mathrm{I}_0$ to denote the identity operator in $H^1_0([a,b))$ as well as $\Upsilon_0: H^1_0([a,b)) \rightarrow H^1_0([a,b))$ for the integral operator 
\begin{align}  
 \Upsilon_0 g(x) = \int_a^b K_0(x,s) g(s)d\dip(s), \quad x\in[a,b),~ g\in H^1_0([a,b)),
\end{align}
which is bounded in view of Theorem~\ref{thm:T0inv}. 
 
 The quadratic operator pencil $\Pe_0$ is closely related to the self-adjoint linear relation $\T_0$. 
 Reminiscent of this fact is the following result (compare Theorem~\ref{th:Resab}). 
 
  \begin{theorem}\label{thmPencilab}
  The spectrum of the quadratic operator pencil $\Pe_0$ coincides with the spectrum of the self-adjoint linear relation $\T_0$. 
  If some $z\in\C$ belongs to the resolvent set of $\Pe_0$, then 
  \begin{align}
    \Pe_0(z)^{-1}g(x) & = \spr{g}{G(x,\cdot\,)^\ast}_{H^1([a,b))}, \quad x\in[a,b), ~g\in H^1_0([a,b)),  
  \end{align}
   where the Green's function $G$ is given by  
 \begin{align}
  G(x,s) = \frac{1}{W(\psi,\phi)} \begin{cases} \psi(x) \phi(s), & s\leq x, \\ \psi(s) \phi(x), & s> x,    \end{cases}
 \end{align}
 and $\psi$, $\phi$ are linearly independent solutions of the homogeneous differential equation~\eqref{eqnDEho} which satisfy $\psi(b) = \phi(a) = 0$.
  \end{theorem}
 
 \begin{proof}
  The claim follows from Theorem~\ref{thm:T0inv} and the Frobenius--Schur factorization (see, for example, \cite[Proposition~1.6.2]{tr08}). 
  More precisely, if $\mathrm{P}$ denotes the projection $\mathrm{P}: \Hoab\rightarrow H_0^1([a,b))$, then one gets from~\eqref{eqnT0invBlock} in Theorem~\ref{thm:T0inv} that  
  \begin{align*}
   \Pe_0(z)^{-1} = \mathrm{P} \left( z \left(\T_0-z\right)^{-1} + \mathrm{I}_{\cH}\right)\mathrm{P}^\ast, \quad z\in\C,  
  \end{align*}
  where $\mathrm{I}_{\cH}$ is the identity in $\Hoab$. 
  This already implies that the resolvent set of $\T_0$ is contained in the resolvent set of $\Pe_0$. 
  For the converse, one simply notes that 
  \begin{align*}
   \left(\T_0-z\right)^{-1} =  \begin{pmatrix} \mathrm{I}_0 & 0 \\ z\,\mathrm{I}_0 & \mathrm{I}_\dip \end{pmatrix} \begin{pmatrix}  \Pe_0(z)^{-1} & 0 \\ 0 & \mathrm{I}_\dip \end{pmatrix} \begin{pmatrix} \Omega_0 + z\, \Upsilon_0  & \Upsilon_0  \\ \mathrm{I}_0 & 0 \end{pmatrix}, \quad z\in\C, 
  \end{align*}
  where $\mathrm{I}_{\dip}$ is the identity in $L^2([a,b);\dip)$. 
  Finally, the representation for the inverse of $\Pe_0(z)$ when $z$ belongs to the resolvent set of $\Pe_0$ follows immediately from Theorem~\ref{th:Resab} (to be precise, the case when $z$ is zero is obtained by direct calculation). 
 \end{proof}
 
 The spectrum of $\Pe_0$ is purely discrete and consists precisely of those $z\in\C$ for which there is a non-trivial solution $\phi$ of the homogeneous differential equation~\eqref{eqnDEho} which satisfies the Dirichlet boundary conditions $\phi(a) = \phi(b) =0$. 
 In this case, the kernel of $\Pe_0(z)$ is spanned by the function $\phi$ restricted to $[a,b)$.

\subsection{Weyl--Titchmarsh functions on a bounded interval}\label{subsecWTbound}

We conclude this section by introducing Weyl--Titchmarsh functions associated with the self-adjoint linear relations $\T_{\alpha,\beta}$. 
To this end, let $\theta_\alpha(z,\redot)$, $\phi_{\alpha}(z,\redot)$ be the solutions of the homogeneous differential equation~\eqref{eqnDEho} with the initial conditions 
 \begin{align}
\phi_{\alpha}(z,a)&=\sin\alpha, & \phi_{\alpha}'(z,a)&=z\cos\alpha, \\
\theta_{\alpha}(z,a)&=\cos\alpha, & \theta_{\alpha}'(z,a)&=-z\sin\alpha,
\end{align}
 for every $z\in\C$. 
The complex-valued function $m_{\alpha,\beta}$ is now defined on 
$\rho(\T_{\alpha,\beta})\backslash\lbrace 0\rbrace$ by requiring that the function  
\begin{align}\label{eqnWTdefab}
   \theta_\alpha(z,x) + m_{\alpha,\beta}(z) \phi_\alpha(z,x), \quad x\in\R, 
\end{align} 
satisfies the boundary condition~\eqref{eqnBCb} at $b$ for every $z\in\rho(\T_{\alpha,\beta})\backslash\lbrace 0\rbrace$. 
It is called the Weyl--Titchmarsh function associated with the self-adjoint linear relation $\T_{\alpha,\beta}$.

If $\psi_\beta(z,\redot)$ is a non-trivial solution of the homogeneous differential equation~\eqref{eqnDEho} satisfying the boundary condition~\eqref{eqnBCb} at $b$, then the function $m_{\alpha,\beta}$ is given by 
\begin{align}\label{eq:wf_ab}
m_{\alpha,\beta}(z) = \frac{W(\theta_\alpha,\psi_\beta)(z)}{W(\psi_\beta,\phi_\alpha)(z)} =\frac{z\psi_\beta(z,a)\sin\alpha+\psi_\beta'(z,a)\cos\alpha }{z\psi_\beta(z,a)\cos\alpha-\psi_\beta'(z,a)\sin\alpha},\quad z\in \rho(\T_{\alpha,\beta})\backslash\lbrace 0\rbrace. 
\end{align}
Note that in this case, $\psi_\beta(z,\redot)$ is a scalar multiple of the function in~\eqref{eqnWTdefab}. 

\begin{lemma}\label{lem:R-f}
The function $m_{\alpha,\beta}$ is a Herglotz--Nevanlinna function with 
\begin{align}\label{eqnmHNab}
  m_{\alpha,\beta}(z)^\ast = m_{\alpha,\beta}(z^\ast), \quad z\in\rho(\T_{\alpha,\beta})\backslash\lbrace 0\rbrace.
\end{align}
\end{lemma}

\begin{proof}
 We choose $\psi_\beta(z,\redot)$ to be the solution of~\eqref{eqnDEho} with the initial values 
 \begin{align*}
  \psi_\beta(z,b) & = \sin\beta, & \psi'(z,b) & = z\cos\beta, 
 \end{align*}
 for every $z\in\C$.
 Since the functions $\psi_\beta(\ledot,a)$ and $\psi_\beta'(\ledot,a)$ are real entire by Lemma~\ref{lemEE} and Lemma~\ref{lemSolEnt}, we infer from~\eqref{eq:wf_ab} that $m_{\alpha,\beta}$ is analytic on $\C\backslash\R$ as well as that the relation~\eqref{eqnmHNab} holds. 
 An evaluation of the expression
 \begin{align*}
  \frac{\psi_\beta'(z,a)}{z\psi_\beta(z,a)} - \left(\frac{\psi_\beta'(z,a)}{z \psi_\beta(z,a)}\right)^\ast = \frac{z^\ast \psi_\beta(z^\ast,a)\psi_\beta'(z,a) - z\psi_\beta(z,a) \psi_\beta'(z^\ast,a)}{|z\psi_\beta(z,a)|^2}, \quad z\in\C\backslash\R,
 \end{align*} 
 using the Lagrange identity shows that $m_{0,\beta}$ is a Herglotz--Nevanlinna function. 
 In order to conclude the proof, one notes that in view of~\eqref{eq:wf_ab}, the function $m_{\alpha,\beta}$ is a composition of $m_{0,\beta}$ with a linear fractional transform which is a Herglotz--Nevanlinna function itself.  
\end{proof}

\begin{remark}
 It is easy to see that $m_{\alpha,\beta}$ is a Weyl function in the sense of \cite{dema91}. 
 In particular, this immediately implies that it is a Herglotz--Nevanlinna function. 
\end{remark}

As the Wronski determinant of two solutions is constant, we get from~\eqref{eq:wf_ab} that 
\begin{align}\label{eqnmalbe2}
 m_{\alpha,\beta}(z) = - \frac{z\theta_\alpha(z,b)\cos\beta - \theta_\alpha'(z,b)\sin\beta}{z\phi_\alpha(z,b)\cos\beta - \phi_\alpha'(z,b)\sin\beta}, \quad z\in\rho(\T_{\alpha,\beta})\backslash\lbrace 0\rbrace. 
\end{align}
This shows that the nonzero spectrum of $\T_{\alpha,\beta}$ coincides with the nonzero poles of the meromorphic function $m_{\alpha,\beta}$.  
Although a pole of $m_{\alpha,\beta}$ at zero always requires zero to be an eigenvalue of $\T_{\alpha,\beta}$, the converse fails if and only if $\alpha\not=0$ and $\beta=0$.  

\begin{corollary}\label{corWTasymp}
  The functions $m_{\alpha,\beta}$ have the asymptotics  
   \begin{align}\label{eqnAsym}
    m_{\alpha,\beta}(z) & = \begin{cases} -\frac{\Lambda_\beta}{2z} + \OO(1), & \alpha=0, \\  - \cot\alpha + \frac{2z}{\Lambda_\beta} \frac{1}{\sin^{2}\alpha} + \OO(z^2), & \alpha\in(0,\pi), \end{cases} \quad z\rightarrow0,  
 \end{align}
  where the positive constants $\Lambda_\beta$ are given by 
 \begin{align}
  \Lambda_0 & =  \coth\left(\frac{b-a}{2}\right), & \Lambda_\beta & = \tanh\left(\frac{b-a}{2}\right), \quad \beta\in(0,\pi). 
 \end{align}
\end{corollary}

\begin{proof}
 In view of~\eqref{eq:wf_ab}, it suffices to prove the claim when $\alpha=0$. 
 Noting that  
 \begin{align*}
  \theta_0(z,b) & = \cosh\left(\frac{b-a}{2}\right) + \OO(z), & \theta_0'(z,b) & = \frac{1}{2} \sinh\left(\frac{b-a}{2}\right) + \OO(z), \\
  \phi_0(z,b) & = 2 z \sinh\left(\frac{b-a}{2}\right) + \OO(z^2), & \phi_0'(z,b) & = z \cosh\left(\frac{b-a}{2}\right) + \OO(z^2), 
 \end{align*} 
 as $z\rightarrow0$, the claim is readily verified in this case. 
\end{proof}

Finally, the results of the present section also allow us to provide a particular growth restriction for the solutions $\theta_\alpha$ and $\phi_\alpha$ in an effortless manner.  

\begin{corollary}\label{corFScartwright}
The real entire functions 
\begin{align}\label{eqnphiathetaaent}
  z & \mapsto\phi_\alpha(z,b), & z & \mapsto \phi'_\alpha(z,b), & z & \mapsto\theta_\alpha(z,b), & z&\mapsto\theta'_\alpha(z,b),    
\end{align}
belong to the Cartwright class.
\end{corollary}

\begin{proof}
 First of all, we note the simple identity  
\begin{align}\label{eqnphiCW}
  \phi_\alpha(z,b)^2 =\frac{\phi_\alpha(z,b)}{\theta_\alpha(z,b)} \left(\frac{\phi_\alpha'(z,b)}{z\phi_\alpha(z,b)}-\frac{\theta_\alpha'(z,b)}{z\theta_\alpha(z,b)}\right)^{-1}, \quad z\in\C\backslash\R.  
\end{align}
In view of~\eqref{eqnmalbe2}, the first quotient in this equation is a Herglotz--Nevanlinna function. 
Similarly to the proof of Lemma \ref{lem:R-f}, one shows that also both of the quotients in the brackets are anti-Herglotz--Nevanlinna functions. 
Consequently, the function $\phi_\alpha(\ledot,b)^2$ is of bounded type in the upper and in the lower complex half-plane \cite[Problem~20]{dB68}. 
Applying a theorem by Krein \cite[Theorem~6.17]{roro94}, \cite[Section~16.1]{le96}, we see that this function belongs to the Cartwright class and hence so does $\phi_\alpha(\ledot,b)$.

In order to finish the proof, one notes that the functions in~\eqref{eqnphiathetaaent} are of bounded type in the upper and in the lower complex half-plane, which follows from the fact that all the quotients in~\eqref{eqnphiCW} are (anti) Herglotz--Nevanlinna functions.  
\end{proof}

With this result at hand, it is now not hard to improve on Lemma~\ref{lemSolEnt} for solutions of the homogeneous differential equation~\eqref{eqnDEho}, that is, when $\chi=0$. 
In fact, along the ideas of the proof of Corollary~\ref{corFScartwright} one can show that the entire functions in~\eqref{eq:fz} belong to the Cartwright class for every $x\in\R$ in this case. 

\begin{corollary}\label{cor:S_p}
    The resolvent of every self-adjoint extension of $\Tmin$ belongs to the $p$-th Schatten class for every $p>1$. 
    If $\Tmin$ is semi-bounded from above or from below, then they even belong to the trace class. 
\end{corollary}

\begin{proof}
 Since the deficiency indices of $\Tmin$ equal two, it suffices to prove the claim for any self-adjoint realization $\T_{\alpha,\beta}$. 
 Noting that the spectrum of $\T_{\alpha,\beta}$ coincides with the set of zeros of the entire function 
  \begin{align*}
  z \phi_{\alpha}(z,b) \cos\beta - \phi_\alpha'(z,b) \sin\beta, \quad z\in\C, 
 \end{align*}
  the claim follows from \cite[Theorem~17.2.1]{le96} and Corollary~\ref{corFScartwright}.
 \end{proof}

\section{The spectral problem on a semi-axis}\label{sec:04}

We will next consider the spectral problem on a semi-axis of the form $J_{+} = [c,\infty)$ or $J_{-} = (-\infty,c)$ for some point $c\in\R$. 
The corresponding Hilbert space is 
\begin{align}
 \cH(J_{\pm}) = H^1(J_{\pm}) \times  L^2(J_{\pm};\dip),
\end{align} 
 equipped with the scalar product
\begin{align}\begin{split}
 \spr{f}{g}_{\cH(J_{\pm})} = \frac{1}{4} \int_{J_{\pm}} f_1(x) g_1(x)^\ast & dx  + \int_{J_{\pm}} f_1'(x) g_1'(x)^\ast dx \\
                            &  + \int_{J_{\pm}} f_2(x) g_2(x)^\ast d\dip(x), \quad f,\,g\in\cH(J_{\pm}). 
\end{split}\end{align} 
Apart from this, we again introduce the closed linear subspace   
\begin{align} 
 \cH_0(J_{\pm}) = H_0^1(J_{\pm}) \times L^2(J_{\pm};\dip).
\end{align}
 Clearly, point evaluations of the first component are continuous on $\cH(J_\pm)$. 
 For each $x$ in the closure of $J_\pm$, we denote with $\delta_{x,\pm}$ the function in $\cH(J_\pm)$ such that 
 \begin{align}\label{eqnPEdeltacpm}
  \spr{f}{\delta_{x,\pm}}_{\cH(J_\pm)} = f_1(x), \quad f\in\cH(J_\pm). 
 \end{align}
 It is readily verified that this function is simply given by 
 \begin{align}
  \delta_{x,\pm}(s) = \begin{pmatrix} 1 \\ 0 \end{pmatrix}  \begin{cases} \E^{\mp\frac{x-c}{2}} \,2\cosh\left(\frac{s-c}{2}\right), & s\lessgtr x, \\  \E^{\mp\frac{s-c}{2}} \,2\cosh\left(\frac{x-c}{2}\right), & s\gtrless x.    \end{cases}
 \end{align}

\subsection{Self-adjointness of the spectral problem on a semi-axis}

As in Subsection~\ref{ssec31}, the maximal relation $\Tmaxpm$ in $\cH(J_{\pm})$ is defined by restricting $\Tloc$;
\begin{align}\label{eq:4.01}
 \Tmaxpm = \left\lbrace \f\in\cH(J_{\pm})\times\cH(J_{\pm}) \left|\, \f=\g|_{J_{\pm}} \,\text{for some }\g\in\Tloc \right.\right\rbrace.
\end{align}
 Given $\f\in\Tmaxpm$, we will again identify it with any representative in $\Tloc$. 
 In this respect, one notes that the quantities $f_1(x)$, $f_1'(x)$, $\tau f_1(x)$ as well as $V(\f,\g)(x)$ are well-defined and the Lagrange identity~\eqref{eq:lag} holds for all $x$, $y$ in the closure of $J_{\pm}$. 
  
\begin{lemma}\label{lem:4.02}
 For every $\f$, $\g\in \Tmaxpm$ we have 
 \begin{align}\label{eq:V=0}
 \lim_{x\rightarrow\pm\infty} V(\f,\g)(x) = 0. 
 \end{align}
 \end{lemma}
 
 \begin{proof}
It follows from the Lagrange identity \eqref{eq:lag} that the limit in \eqref{eq:V=0} exists. 
Since the function $V(\f,\g)$ is integrable near $\pm\infty$, the limit has to be zero.
 \end{proof}

  For a given $z\in\C$, we say that a solution $\psi_\pm$ of the homogeneous differential equation~\eqref{eqnDEho} lies in $\HR$ near $\pm\infty$ if it lies in $H^1(\R)$ near $\pm\infty$ and $z\psi_\pm$ lies in $L^2(\R;\dip)$ near $\pm\infty$.
 With this definition, the kernel of $\Tmaxpm-z$ can be identified with the subspace of all solutions of the homogeneous differential equation~\eqref{eqnDEho} which lie in $\HR$ near $\pm\infty$. 
 In contrast to the situation on a bounded interval, not all solutions are represented in this kernel. 
 The following result will show that the dimension of the kernel of $\Tmaxpm-z$ is indeed at most one, that is, 
 \begin{align}\label{eqnkerpmleq1}
  \dim\ker{\Tmaxpm-z} \leq 1. 
 \end{align}

\begin{corollary}\label{cor:4.03}
 For every $z\in\C$ there is a solution $\theta_\pm$ of the homogeneous differential equation~\eqref{eqnDEho} which does not lie in $\HR$ near $\pm\infty$. 
\end{corollary}

\begin{proof}
 Since the claim is obvious when $z$ equals zero, we may assume that $z$ is non-zero. 
 Any two solutions $\psi_\pm$, $\phi_\pm$ of the homogeneous differential equation~\eqref{eqnDEho} which lie in $\HR$ near $\pm\infty$ correspond to elements of the kernel of $\Tmaxpm-z$. 
 In view of~\eqref{eqnVtoW} and Lemma \ref{lem:4.02}, it turns out that their Wronski determinant $W(\psi_\pm,\phi_\pm)$ is equal to zero and thus $\psi_\pm$ and $\phi_\pm$ are linearly dependent.
\end{proof}
 
As the linear relation $\Tmaxpm$ will turn out not to be symmetric, we introduce  
\begin{align}\label{eq:4.02}
\Tminpm = \left\lbrace \f\in \Tmaxpm \,|\, \tau f_1(c)=f_1'(c)=0\right\rbrace,
\end{align}
which is referred to as the minimal relation. 

\begin{theorem}\label{th:4.01}
  The minimal relation $\Tminpm$ is symmetric in $\cH(J_{\pm})$ with  
 \begin{align}\label{eqnTminTmaxpm}
  \Tminpm^\ast &= \Tmaxpm, &  \Tmaxpm^\ast &= \Tminpm. 
  \end{align}
In particular, the linear relations $\Tminpm$ and $\Tmaxpm$ are closed. 
\end{theorem}

\begin{proof}
Employing the Lagrange identity \eqref{eq:lag} and Lemma~\ref{lem:4.02}, we get
\begin{align}\label{eq:4.05}
\spr{f}{\tau g}_{\cH(J_{\pm})}- \spr{\tau f}{g}_{\cH(J_{\pm})}= \pm V(\f,\g^\ast)(c)
\end{align} 
for all $\f$, $\g\in\Tmaxpm$. 
As in the proof of Theorem~\ref{th:3.01} this implies 
\begin{align*}
 \Tmaxpm & \subseteq \Tminpm^\ast, &  \Tminpm & \subseteq\Tmaxpm^\ast. 
\end{align*}
In order to show that $\Tminpm^\ast\subseteq\Tmaxpm$, fix some $(f,f_\tau)\in\Tminpm^\ast$ and let $\h\in\Tloc$ such that $\tau h$ coincides with $f_\tau$ on $J_\pm$. 
For any $\g\in\Tminpm$ we clearly have  
\begin{align*}
 \spr{f_\tau}{g}_{\cH(J_\pm)} = \spr{f}{\tau g}_{\cH(J_\pm)}.
\end{align*}
If $\tau g$ vanishes near $\pm\infty$, then integrating the left-hand side by parts, using that $\g\in\Tminpm$, noting that $f_\tau$ coincides with $\tau h$ on $J_\pm$ and integrating by parts once more, using that $\h\in\Tloc$, we end up with 
\begin{align*}
  \frac{1}{4} \int_{J_\pm} (h_1(x)-f_1(x))\tau g_1(x)^\ast & dx + \int_{J_\pm} (h_1'(x)-f_1'(x))\tau g_1'(x)^\ast dx \\ 
                                                              & + \int_{J_\pm} (h_2(x)-f_2(x)) \tau g_2(x)^\ast d\dip(x) = 0. 
\end{align*}
Upon choosing the constants in Corollary~\ref{cor:inhom} appropriately, one sees that the range of $\Tminpm$ actually contains all functions in $\cH_0(J_\pm)$ which vanish near $\pm\infty$; see also \eqref{eq:ran_oab}. 
Hence we infer that $h_2(x) = f_2(x)$ for almost all $x\in J_\pm$ with respect to $\dip$ and moreover, from~\eqref{eq:ran_ab}  we see that $h_1 - f_1$ coincides with a solution of the homogeneous differential equation~\eqref{eqnDEho} with $z=0$ on (every subinterval of) $J_\pm$. 
Altogether, this shows that $(f,f_\tau)$ has a representative in $\Tloc$ and hence belongs to $\Tmaxpm$, which proves the first equality in~\eqref{eqnTminTmaxpm}. 
 In particular, $\Tmaxpm$ turns out to be closed and the second equality in~\eqref{eqnTminTmaxpm} follows upon noting that $\Tminpm$ is closed as well, since the linear functionals $\f\mapsto\tau f_1(c)$ and $\f\mapsto f_1'(c)$ are bounded on $\Tmaxpm$ (cf.\ \cite[Lemma~2.6]{LeftDefiniteSL}). 
\end{proof}

\begin{corollary}\label{cor:4.04}
 The minimal relation $\Tminpm$ has deficiency indices equal to one. 
\end{corollary} 

\begin{proof}
 Since the minimal relation $\Tminpm$ is real with respect to the natural conjugation, the deficiency indices are the same in the upper and lower complex half-plane. 
 In view of~\eqref{eqnkerpmleq1}, it remains to note that $\Tmaxpm$ does not coincide with $\Tminpm$. 
\end{proof}

\begin{corollary}\label{corWS}
 If $z\in\C$ is a point of regular type for $\Tminpm$, then there is a (up to scalar multiples) unique non-trivial solution $\psi_\pm$ of the homogeneous differential equation~\eqref{eqnDEho} which lies in $\HR$ near $\pm\infty$.
\end{corollary} 

\begin{proof}
 It suffices to note that the solutions of the homogeneous differential equation~\eqref{eqnDEho} which lie in $\HR$ near $\pm\infty$ correspond to the kernel of $\Tmaxpm-z$, which is one-dimensional by Corollary~\ref{cor:4.04}.
\end{proof}

Since the minimal linear relation $\Tminpm$ has equal deficiency indices, it always has self-adjoint extensions. 
Using the Lagrange identity~\eqref{eq:4.05}, we readily obtain a complete description of all of them in a standard way: 
For all $\gamma\in[0,\pi)$ we set 
 \begin{align}\label{eq:T_a}
   \T_{\gamma,\pm} = \left\lbrace \f\in \Tmaxpm \,|\,  \tau f_1(c) \cos\gamma - f_1'(c) \sin\gamma = 0 \right\rbrace.
 \end{align}
 The eigenvalues of the self-adjoint linear relation $\T_{\gamma,\pm}$ are precisely those $z\in\C$ for which there is a non-trivial solution $\phi$ of the homogeneous differential equation~\eqref{eqnDEho} which lies in $\HR$ near $\pm\infty$ and satisfies the boundary condition 
\begin{align}\label{eqnBCc}  
   z f(c) \cos\gamma - f'(c)\sin\gamma & =0, 
\end{align}
at $c$. 
 In view of~\eqref{eqnkerpmleq1}, every eigenvalue of $\T_{\gamma,\pm}$ turns out to be simple. 
 This even holds for zero in this case, which is an eigenvalue of $\T_{\gamma,\pm}$ if and only if $\gamma$ is zero. 

\begin{theorem}\label{thm:Respm}
 If some nonzero $z\in\C$ belongs to the resolvent set of $\T_{\gamma,\pm}$, then   
  \begin{align}\begin{split}\label{eq:Respm}
   z\, (\T_{\gamma,\pm}-z)^{-1} g(x) & = \spr{g}{\G_\pm(x,\cdot\,)^\ast}_{\cH(J_{\pm})} \begin{pmatrix} 1 \\ z \end{pmatrix} - g_1(x) \begin{pmatrix} 1 \\ 0 \end{pmatrix}, \quad x\in J_{\pm},  
 \end{split}\end{align}
 for every $g\in\cH(J_{\pm})$, 
where the Green's function $\G_\pm$ is given by  
 \begin{align}
  \G_\pm(x,s) =  \begin{pmatrix} 1 \\ z \end{pmatrix} \frac{\pm 1}{W(\psi_\pm,\phi)} \begin{cases} \psi_\pm(x) \phi(s), & s\lessgtr x, \\ \psi_\pm(s) \phi(x), & s \gtrless x,    \end{cases}
 \end{align}
 and $\psi_\pm$, $\phi$ are linearly independent solutions of the homogeneous differential equation~\eqref{eqnDEho} such that $\psi_\pm$ lies in $\HR$ near $\pm\infty$ and $\phi$ satisfies the boundary condition~\eqref{eqnBCc} at $c$. 
 \end{theorem}

\begin{proof}
 First of all, note that non-trivial solutions $\psi_\pm$, $\phi$ of the homogeneous differential equation~\eqref{eqnDEho} with the required properties always exist. 
 Moreover, they are linearly independent since otherwise $z$ would be an eigenvalue of $\T_{\gamma,\pm}$. 
 If some pair $(f,g)$ belongs to $\T_{\gamma,\pm} - z$, then one shows as in the proof of Theorem~\ref{th:Resab} 
\begin{align*}
  z f_1(x) & =  \spr{g}{\G_\pm(x,\redot)^\ast}_{\cH(J_{r,\pm})} - g_1(x) \\
            & \quad\quad + \frac{\phi(x)}{W(\psi_\pm,\phi)} \left( z\psi_\pm(r)f_1'(r) - \psi_\pm'(r)(zf_1(r)+g_1(r)) \right), \quad x\in J_{\pm}, 
\end{align*}
at least if $\pm\, r > |x|$, were we use the abbreviations $J_{r,+} = [c,r)$ and $J_{r,-} = [r,c)$.  
In order to obtain the representation for $f_1$ as given in~\eqref{eq:Respm}, it remains to note that the last term on the right-hand side converges to zero as $r\rightarrow\pm\infty$ in view of Lemma~\ref{lem:4.02}.  
Finally, as an immediate consequence of the definition of $\T_{\gamma,\pm}$ we have 
\begin{align*}
 f_2(x)  = g_1(x)+ z f_1(x)= \spr{g}{\G_\pm(x,\redot)^\ast}_{\cH(J_{\pm})},
\end{align*}
for almost all $x\in J_{\pm}$ with respect to $\dip$. 
\end{proof}

\subsection{A quadratic operator pencil on a semi-axis}

We will now introduce a quadratic operator pencil in $H_0^1(J_{\pm})$, associated with the differential equation~\eqref{eqnDEinho} and a Dirichlet boundary condition at $c$. 
To this end, we first consider the linear relation $\T_\pm$ in $\cH_0(J_{\pm})$, defined by  
 \begin{align}\label{eq:Tpm}
  \T_\pm = \left\lbrace \f\in\T_{0,\pm} \,|\, \f\in\cH_0(J_{\pm})\times\cH_0(J_{\pm}) \right\rbrace.
 \end{align}
 Upon observing that $\ker{\T_{0,\pm}}$ is spanned by the single function $\E^{\mp\frac{x}{2}}$ on $J_\pm$, it is readily verified that $\cH(J_{\pm}) = \cH_0(J_{\pm}) \oplus \ker{\T_{0,\pm}}$, and therefore one finds   
 \begin{align}
  \T_{0,\pm} = \T_\pm \oplus \left(\ker{\T_{0,\pm}}\times\lbrace0\rbrace\right).
 \end{align} 
 This guarantees that the linear relation $\T_\pm$ is self-adjoint in the Hilbert space $\cH_0(J_{\pm})$ with $\sigma(\T_{0,\pm}) = \sigma(\T_{\pm})\cup\lbrace0\rbrace$ and that zero is not an eigenvalue of $\T_\pm$. 

 \begin{theorem}\label{thm:Tinv_pm}
  If zero belongs to the resolvent set of $\T_{\pm}$, then   
  \begin{align}
   \T_{\pm}^{-1} = \begin{pmatrix} \Omega_\pm & \Upsilon_\pm  \\ \mathrm{I}_\pm & 0 \end{pmatrix},  
  \end{align}
 where the operator $\Omega_\pm: H_0^1(J_{\pm})\rightarrow H_0^1(J_{\pm})$ is given by 
\begin{align}\label{eqnIOomegapm}
 \Omega_\pm g_1(x) =  \int_{J_{\pm}} K_\pm(x,s) g_1(s)d\omega(s), \quad x\in J_{\pm},~ g_1\in H_{\mathrm{c}}^1(J_{\pm}),
\end{align}
the operator $\Upsilon_\pm: L^2(J_{\pm};\dip) \rightarrow H_0^1(J_{\pm})$ is given by 
\begin{align}\label{eqnIOdippm}
 \Upsilon_\pm g_2(x) = \int_{J_{\pm}} K_\pm(x,s) g_2(s)d\dip(s), \quad x\in J_{\pm},~ g_2\in L^2(J_{\pm};\dip),
\end{align}
 and the operator $\mathrm{I}_\pm: H_0^1(J_{\pm}) \rightarrow L^2(J_{\pm};\dip)$ is the canonical embedding. Here,  
  \begin{align}
   K_\pm(x,s) =  \pm 2 \begin{cases} \E^{\mp\frac{x-c}{2}} \sinh\left(\frac{s-c}{2}\right), & s\lessgtr x,   \\
                                                    \E^{\mp\frac{s-c}{2}} \sinh\left(\frac{x-c}{2}\right), & s\gtrless x.  \end{cases}
 \end{align} 
 \end{theorem}

 \begin{proof}
  Suppose that zero belongs to the resolvent set of $\T_\pm$ and let $(f,g)\in\T_\pm$. 
  By the definition of $\T_{\pm}$ one has $f_2(x) = g_1(x)$ for almost all $x\in J_{\pm}$ with respect to $\dip$, which ensures that the canonical embedding $\mathrm{I}_\pm$ is bounded (as the inverse of $\T_{\pm}$ is as well) and that the right-hand side of~\eqref{eqnIOdippm} always exists.  
  In much the same manner as in the proof of Theorem~\ref{th:Resab} and Theorem~\ref{thm:T0inv}, using~\eqref{eqnDErel}, integration by parts and the boundary condition at $c$, one obtains 
  \begin{align*}
  f_1(x) & = \pm\int_{c}^r  K_\pm(x,s)g_1(s) d\omega(s) \pm  \int_{c}^r  K_\pm(x,s) g_2(s) d\dip(s) \\ 
                 & \qquad\qquad\qquad + 2 \sinh\left(\frac{x-c}{2}\right) \E^{\mp\frac{r-c}{2}} \left( f_1'(r) \pm\frac{1}{2}f_1(r) \right), \quad x\in J_{\pm}, 
  \end{align*}
  at least when $\pm\, r>|x|$. 
  If $g_1$ has compact support, then one shows similarly to the proof of Lemma~\ref{lem:4.02} that the last term converges to zero as $r\rightarrow\pm\infty$. 
  This yields  
  \begin{align*}
    f_1(x) = \int_{J_{\pm}} K_\pm(x,s) g_1(s) d\omega(s) & + \int_{J_{\pm}} K_\pm(x,s) g_2(s) d\dip(s), \quad x\in J_{\pm}, 
  \end{align*}
  which proves the claimed representation for the inverse of $\T_\pm$. 
 \end{proof}
 
 Supposing that zero belongs to the resolvent set of $\T_{\pm}$, we now define the quadratic operator pencil $\Pe_\pm$ in $H^1_0(J_{\pm})$ by setting 
\begin{align}
 \Pe_\pm(z) = \mathrm{I}_\pm - z\, \Omega_\pm - z^2 \Upsilon_\pm, \quad z\in\C.
\end{align}
Again, by abuse of notation, we reuse the symbol $\mathrm{I}_\pm$ to denote the identity operator in $H^1_0(J_{\pm})$ as well as $\Upsilon_\pm: H^1_0(J_{\pm}) \rightarrow H^1_0(J_{\pm})$ for the integral operator 
\begin{align}  
 \Upsilon_\pm g(x) = \int_{J_{\pm}} K_\pm(x,s) g(s)d\dip(s), \quad x\in J_{\pm},~ g\in H^1_0(J_{\pm}),
\end{align}
which is bounded in view of Theorem~\ref{thm:Tinv_pm}. 
 
 It is not surprising that the quadratic operator pencil $\Pe_\pm$ is again closely related to the self-adjoint linear relation $\T_\pm$. 
 The proof of the following result (compare Theorem~\ref{thm:Respm}) is almost literally the same as the one for Theorem~\ref{thmPencilab}.  
  
  \begin{theorem}\label{thmPencilpm}
  Suppose that zero belongs to the resolvent set of $\T_\pm$.
  The  spectrum of the quadratic operator pencil $\Pe_\pm$ coincides with the spectrum of the self-adjoint linear relation $\T_\pm$. 
  If some $z\in\C$ belongs to the resolvent set of $\Pe_\pm$, then 
  \begin{align}
    \Pe_\pm(z)^{-1}g(x) & = \spr{g}{G_\pm(x,\cdot\,)^\ast}_{H^1(J_{\pm})}, \quad x\in J_{\pm}, ~g\in H^1_0(J_{\pm}),  
  \end{align}
   where the Green's function $G_\pm$ is given by  
 \begin{align}
  G_\pm(x,s) = \frac{\pm 1}{W(\psi_\pm,\phi)} \begin{cases} \psi_\pm(x) \phi(s), & s\lessgtr x, \\ \psi_\pm(s) \phi(x), & s\gtrless x,    \end{cases}
 \end{align}
 and $\psi_\pm$, $\phi$ are linearly independent solutions of the homogeneous differential equation~\eqref{eqnDEho} such that $\psi_\pm$ lies in $H^1(\R)$ near $\pm\infty$ and $\phi$ satisfies $\phi(c) = 0$.
  \end{theorem}
 
 The eigenvalues of $\Pe_\pm$ consist precisely of those $z\in\C$ for which there is a non-trivial solution $\psi_\pm$ of the homogeneous differential equation~\eqref{eqnDEho} which lies in $H^1(\R)$ near $\pm\infty$ and satisfies the Dirichlet boundary condition $\psi_\pm(c) =0$. 
 In this case, the kernel of $\Pe_\pm(z)$ is spanned by the function $\psi_\pm$ restricted to $J_{\pm}$.

\subsection{Weyl--Titchmarsh functions on a semi-axis}\label{subsecWTpm}

In order to introduce a Weyl--Titchmarsh function associated with the self-adjoint linear relation $\T_{\gamma,\pm}$, we let $\theta_{\gamma}(z,\redot)$, $\phi_{\gamma}(z,\redot)$ be the solutions of the homogeneous differential equation~\eqref{eqnDEho} with the initial conditions 
\begin{align}\label{eqnpgammapm}
\phi_{\gamma}(z,c)&=\sin\gamma, & \phi_{\gamma}'(z,c)&=z\cos\gamma,\\ \label{eqntgammapm}
\theta_{\gamma}(z,c)&=\cos\gamma, & \theta_{\gamma}'(z,c)&=-z\sin\gamma,
\end{align}
for every $z\in\C$. 
The complex-valued function $m_{\gamma,\pm}$ is now defined on $\rho(\T_{\gamma,\pm})\backslash\lbrace 0\rbrace$ by requiring that the function 
\begin{align}\label{eq:m_pm}
  \theta_{\gamma}(z,x)\pm m_{\gamma,\pm}(z)\phi_{\gamma}(z,x), \quad x\in\R, 
\end{align}
lies in $\HR$ near $\pm\infty$ for every $z\in\rho(\T_{\gamma,\pm})\backslash\lbrace 0\rbrace$.
 In view of Corollary~\ref{corWS} and the fact that $\phi_\gamma(z,\redot)$ does not lie in $\HR$ near $\pm\infty$ unless $z$ is an eigenvalue of $\T_{\gamma,\pm}$, one notes that this function is well-defined. 
 It is henceforth called the Weyl--Titchmarsh function associated with the self-adjoint linear relation $\T_{\gamma,\pm}$.

If $\psi_\pm(z,\redot)$ is a non-trivial solution of the homogeneous differential equation~\eqref{eqnDEho} which lies in $\HR$ near $\pm\infty$, then the function $m_{\gamma,\pm}$ is given by 
 \begin{align}\label{eq:wf_pm}
  \pm\, m_{\gamma,\pm}(z) = \frac{W(\theta_\gamma,\psi_\pm)(z)}{W(\psi_\pm,\phi_\gamma)(z)}  = \frac{z\psi_\pm(z,c)\sin\gamma + \psi_\pm'(z,c)\cos\gamma}{z\psi_\pm(z,c)\cos\gamma - \psi_\pm'(z,c)\sin\gamma},\quad z\in\rho(\T_{\gamma,\pm})\backslash\lbrace 0\rbrace.
 \end{align} 
Note that in this case, $\psi_\pm(z,\redot)$ is a scalar multiple of the function in~\eqref{eq:m_pm}. 

\begin{lemma}
The function $m_{\gamma,\pm}$ is a Herglotz--Nevanlinna function with 
\begin{align}\label{eqnHNpmsym}
  m_{\gamma,\pm}(z)^\ast = m_{\gamma,\pm}(z^\ast), \quad z\in\rho(\T_{\gamma,\pm})\backslash\lbrace 0\rbrace.
\end{align}
\end{lemma}

\begin{proof}
 From Theorem~\ref{thm:Respm} we infer that 
\begin{align}\begin{split}\label{eqnRelResm}
 z\, & \spr{(\T_{\gamma,\pm}-z)^{-1}\delta_{x,\pm}}{\delta_{x,\pm}}_{\cH(J_{\pm})} +  1 + \E^{\mp(x-c)}  \\ 
    & \qquad= \pm\, z^{-1} \left(\theta_\gamma(z,x) \pm m_{\gamma,\pm}(z) \phi_\gamma(z,x) \right) \phi_\gamma(z,x), \quad z\in\rho(\T_{\gamma,\pm})\backslash\lbrace 0\rbrace.
\end{split}\end{align}
Since the solution $\phi_\gamma(z,\redot)$ neither vanishes identically on $J_{+}$ nor on $J_{-}$ for every nonzero $z\in\C$, this shows that the function $m_{\gamma,\pm}$ is analytic on $\C\backslash\R$ as well as the relation~\eqref{eqnHNpmsym}. 
The fact that $m_{\gamma,\pm}$ is a Herglotz--Nevanlinna function is verified along the lines of Lemma~\ref{lem:R-f}, upon additionally taking into account Lemma~\ref{lem:4.02}. 
\end{proof}

\begin{remark}
 It is easy to see that $m_{\gamma,\pm}$ is a Weyl function in the sense of \cite{dema91}. 
 In particular, this immediately implies that it is a Herglotz--Nevanlinna function. 
\end{remark}

Let us explicitly point out the useful relation 
 \begin{align}\label{eq:mRescon}
 -\frac{1}{m_{0,\pm}(z)} & =  m_{\frac{\pi}{2},\pm}(z) = z^2 \spr{(\T_{\frac{\pi}{2},\pm}-z)^{-1}\delta_{c,\pm}}{\delta_{c,\pm}}_{\cH(J_{\pm})} + 2z, \quad z\in\C\backslash\R, 
 \end{align} 
 which follows immediately from~\eqref{eq:wf_pm} and~\eqref{eqnRelResm} upon setting $x=c$.  

We conclude this section by relating the spectrum of the self-adjoint linear relation $\T_{\gamma,\pm}$ to the singularities of the Weyl--Titchmarsh function $m_{\gamma,\pm}$. 

\begin{lemma}\label{lemTrsMdh}
 The resolvent set of the self-adjoint linear relation $\T_{\gamma,\pm}$ coincides with the maximal domain of holomorphy of the Weyl--Titchmarsh function $m_{\gamma,\pm}$. 
\end{lemma}

\begin{proof}
 It follows from~\eqref{eqnRelResm} that $m_{\gamma,\pm}$ is analytic on the resolvent set of  $\T_{\gamma,\pm}$ excluding zero. 
 Moreover, if zero belongs to the resolvent set of $\T_{\gamma,\pm}$, then $\gamma$ is non-zero and choosing $x=c$ shows that $m_{\gamma,\pm}$ has an analytic extension to zero. 
  
 In order to prove the converse, first of all one observes that for every $f\in\cH(J_{\pm})$ with $f_1 = 0$ and such that $f_2$ vanishes near $\pm\infty$, Theorem~\ref{thm:Respm} yields   
 \begin{align*}
  \spr{(\T_{\gamma,\pm}-z)^{-1}f}{f}_{\cH(J_{\pm})} = H_{f,\pm}(z) + m_{\gamma,\pm}(z) F_{f,\pm}(z), \quad z\in\C\backslash\R, 
 \end{align*} 
 for some entire functions $H_{f,\pm}$ and $F_{f,\pm}$. 
 Here, one should note that $\theta_\gamma$ and $\phi_\gamma$ are locally uniformly bounded on $\C\times\R$ (which can be deduced from the Gronwall lemma \cite[Lemma~1.3]{be89}, \cite[Lemma~A.1]{MeasureSL}). 
 Now if $m_{\gamma,\pm}$ has an analytic extension to a neighborhood of some nonzero $\lambda\in\R$, then this, in conjunction with~\eqref{eqnRelResm}, shows that the function 
 \begin{align*}
  \spr{(\T_{\gamma,\pm}-z)f}{f}_{\cH(J_{\pm})}, \quad z\in\C\backslash\R, 
 \end{align*}
 has an analytic extension to a (fixed) neighborhood of $\lambda$ as well, for all $f$ in a dense subspace of $\cH(J_{\pm})$. 
 But this guarantees that $\lambda$ belongs to the resolvent set of $\T_{\gamma,\pm}$. 
 
  Finally, suppose $m_{\gamma,\pm}$ has an analytic extension to a neighborhood of zero. 
 If zero belonged to the spectrum of $\T_{\gamma,\pm}$, then it would be an isolated point by the above arguments and hence an eigenvalue of $\T_{\gamma,\pm}$, implying that $\gamma$ is zero.
 In view of~\eqref{eq:mRescon}, this gives the contradiction $|m_{\gamma,\pm}(\I\varepsilon)|\rightarrow\infty$ as $\varepsilon\downarrow0$.
\end{proof}

\begin{corollary}\label{corWTasympm}
 The functions $m_{\gamma,\pm}$ have the asymptotics 
 \begin{align}
   m_{\gamma,\pm}(\I\varepsilon) = \begin{cases} - \frac{1}{2\I\varepsilon} (1+\oo(1)), & \gamma = 0, \\ \mp\cot\gamma \pm \frac{2\I\varepsilon}{\sin^2\gamma} + \oo(\varepsilon), & \gamma\in(0,\pi), \end{cases} \quad \varepsilon\downarrow0. 
 \end{align}
\end{corollary}

\begin{proof}
  In view of~\eqref{eq:wf_pm}, it suffices to prove the claim when $\gamma=0$. 
  Noting that zero is not an eigenvalue of $\T_{\frac{\pi}{2},\pm}$, the claim follows from~\eqref{eq:mRescon} in this case. 
\end{proof}

\section{The spectral problem on the whole line}\label{secWL}

It remains to finally discuss spectral theory for the differential equation~\eqref{eqnDEinho} on the whole line. 
To this end, we consider the Hilbert space 
\begin{align}
 \HR = H^1(\R) \times L^2(\R;\dip),
\end{align}
 equipped with the scalar product
\begin{align}\begin{split}
 \spr{f}{g}_{\HR} = \frac{1}{4} \int_{\R} f_1(x) g_1(x)^\ast & dx + \int_{\R} f_1'(x) g_1'(x)^\ast dx \\ & + \int_{\R} f_2(x) g_2(x)^\ast d\dip(x), \quad f,\,g\in\HR.  
\end{split}\end{align}
 Clearly, point evaluations of the first component are continuous on $\HR$. 
 For every $c\in\R$, we denote with $\delta_c$ the function in $\HR$ such that 
 \begin{align}\label{eqnPEdeltac}
  \spr{f}{\delta_c}_{\HR} = f_1(c), \quad f\in\HR. 
 \end{align}
 It is readily verified that this function is simply given by 
 \begin{align}\label{eqnPEdeltacexp}
  \delta_c(x) =  \E^{-\frac{|x-c|}{2}} \begin{pmatrix} 1 \\ 0 \end{pmatrix}, \quad x\in\R. 
 \end{align}

\subsection{Self-adjointness of the spectral problem on the whole line}

We first introduce the linear relation $\T$ in $\HR$ by restricting $\Tloc$ to $\HR\times\HR$; 
\begin{align}
   \T = \left\lbrace \f\in\T_{\loc} \,|\, \f\in\HR\times\HR \right\rbrace.
\end{align}
There is no need to introduce maximal and minimal relations in this case. 

\begin{theorem}\label{thmSAR}
 The linear relation $\T$ is self-adjoint in $\HR$. 
\end{theorem}

\begin{proof} 
 First of all, note that by Lemma~\ref{lem:4.02} we have 
 \begin{align*}
  \lim_{|x|\rightarrow\infty} V(\f,\g)(x) = 0
 \end{align*}
 for every $\f$, $\g\in\T$, 
 from which we infer as in Theorem~\ref{th:4.01} that $\T$ is symmetric. 
 The converse inclusion follows in much the same manner as in Theorem~\ref{th:4.01} upon noting that the range of $\T$ contains all functions in $\HR$ with compact support.  
\end{proof}

 Given some $z\in\C$, we say that a solution $\phi$ of the homogeneous differential equation~\eqref{eqnDEho} lies in $\HR$ if it belongs to $H^1(\R)$ and $z\phi$ belongs to $L^2(\R;\dip)$. 
 With this notation, some $z\in\C$ is an eigenvalue of $\T$ if and only if there is a non-trivial solution $\phi$ of the homogeneous differential equation~\eqref{eqnDEho} which lies in $\HR$. 
 In view of Corollary~\ref{cor:4.03}, every eigenvalue of $\T$ is simple. 
 
 On the other side, the following result shows that the essential spectrum of $\T$ splits into two components, one arising from the left endpoint $-\infty$ and one from the right endpoint $+\infty$. 
 In particular, the essential spectrum of $\T$ is independent of the local behavior of the measures $\omega$ and $\dip$. 

 \begin{lemma}\label{lemSplit}
  For any $c\in\R$, the essential spectrum of $\T$ divides into   
  \begin{align}
   \sigma_{\ess}(\T) = \sigma_{\ess}(\T_-) \cup \sigma_{\ess}(\T_+). 
  \end{align}
 \end{lemma}
 
 \begin{proof}
 To prove the claim, it suffices to note that with respect to the decomposition 
 \begin{align*}
  \HR = \cH_0(J_{-})\oplus\linspan\lbrace\delta_c\rbrace\oplus\cH_0(J_{+}),
 \end{align*}
 the self-adjoint linear relation $\T_- \oplus \left( \lbrace0\rbrace\times\linspan\lbrace\delta_c\rbrace \right) \oplus \T_+$ is a finite dimensional perturbation of $\T$. 
 \end{proof}
 
 
 
 Before we derive a representation for the resolvent of $\T$, let us first also mention the following auxiliary result, which is reminiscent of Corollary~\ref{corWS}. 
 
 \begin{corollary}\label{corTrespsipm}
  If $z\in\C$ belongs to the resolvent set of $\T$, then there is a (up to scalar multiples) unique non-trivial solution $\psi_\pm$ of the homogeneous differential equation~\eqref{eqnDEho} which lies in $\HR$ near $\pm\infty$. 
 \end{corollary}
 
 \begin{proof}
  In view of Corollary~\ref{corWS}, it suffices to show that $z$ is a point of regular type for the linear relation $\Tminpm$ for some $c\in\R$. 
  This is immediate if $z$ belongs to the resolvent set of $\T_{0,\pm}$. 
  Otherwise, Lemma~\ref{lemSplit} implies that $z$ is part of the discrete spectrum of $\T_{0,\pm}$. 
  Now Lemma~\ref{lemTrsMdh} and~\eqref{eq:wf_pm} show that $z$ belongs to the resolvent set of $\T_{\frac{\pi}{2},\pm}$ and hence is a point of regular type for $\Tminpm$. 
 \end{proof}
  
\begin{theorem}\label{thmSARes}
 If some nonzero $z\in\C$ belongs to the resolvent set of $\T$, then  
 \begin{align}\begin{split}\label{eq:Reswl}
  z\, (\T-z)^{-1} g(x) & = \spr{g}{\G(x,\redot)^\ast}_{\HR} \begin{pmatrix} 1 \\ z \end{pmatrix} - g_1(x) \begin{pmatrix} 1 \\ 0 \end{pmatrix}, \quad x\in\R,  
 \end{split}\end{align}
 for every $g\in\HR$, where the Green's function $\G$ is given by  
 \begin{align}
  \G(x,s) = \begin{pmatrix} 1 \\ z \end{pmatrix} \frac{1}{W(\psi_+,\psi_-)} \begin{cases} \psi_+(x) \psi_-(s), & s\leq x, \\ \psi_+(s) \psi_-(x), & s> x,    \end{cases}
 \end{align}
 and $\psi_+$, $\psi_-$ are linearly independent solutions of the homogeneous differential equation~\eqref{eqnDEho} such that $\psi_+$ lies in $\HR$ near $+\infty$ and $\psi_-$ lies in $\HR$ near $-\infty$.
\end{theorem}

\begin{proof}
 First of all, one notes that solutions $\psi_+$, $\psi_-$ of the homogeneous differential equation~\eqref{eqnDEho} with the required properties always exist by Corollary~\ref{corTrespsipm}. 
  Moreover, they are linearly independent since otherwise $z$ would be an eigenvalue of $\T$. 
 Now the claim follows in much the same manner as Theorem~\ref{thm:Respm}.  
\end{proof}

\subsection{A quadratic operator pencil on the whole line}
 
 In this subsection, we will introduce a quadratic operator pencil in $H^1(\R)$, associated with the differential equation~\eqref{eqnDEinho}. 
 First of all, let us mention that zero is not an eigenvalue of $\T$. 
 
 \begin{theorem}\label{thmTinv}
  If zero belongs to the resolvent set of $\T$, then 
  \begin{align}
   \T^{-1} = \begin{pmatrix} \Omega & \Upsilon  \\ \mathrm{I} & 0 \end{pmatrix},  
  \end{align}
 where the operator $\Omega: H^1(\R)\rightarrow H^1(\R)$ is given by 
\begin{align}\label{eqnIOomega}
 \Omega g_1(x) =  \int_{\R} \E^{-\frac{|x-s|}{2}} g_1(s)d\omega(s), \quad x\in\R,~ g_1\in H_{\mathrm{c}}^1(\R),
\end{align}
the operator $\Upsilon: L^2(\R;\dip) \rightarrow H^1(\R)$ is given by 
\begin{align}\label{eqnIOdip}
 \Upsilon g_2(x) = \int_\R \E^{-\frac{|x-s|}{2}} g_2(s)d\dip(s), \quad x\in\R,~ g_2\in L^2(\R;\dip),
\end{align}
 and the operator $\mathrm{I}: H^1(\R) \rightarrow L^2(\R;\dip)$ is the canonical embedding. 
 \end{theorem}

 \begin{proof}
  The claim follows in much the same manner as Theorem~\ref{thm:Tinv_pm}. 
 \end{proof}
 
 Under the additional assumption that zero belongs to the resolvent set of $\T$, we now define the quadratic operator pencil $\Pe$ in $H^1(\R)$ by setting 
\begin{align}
 \Pe(z) = \mathrm{I} - z\, \Omega - z^2 \Upsilon, \quad z\in\C.
\end{align}
Here, by abuse of notation, we reuse the symbol $\mathrm{I}$ to denote the identity operator in $H^1(\R)$ as well as $\Upsilon: H^1(\R) \rightarrow H^1(\R)$ for the integral operator 
\begin{align}  
 \Upsilon g(x) = \int_\R \E^{-\frac{|x-s|}{2}} g(s)d\dip(s), \quad x\in\R,~ g\in H^1(\R),
\end{align}
which is bounded in view of Theorem~\ref{thmTinv}. 
 
  It is not surprising that the quadratic operator pencil $\Pe$ is again closely related to the self-adjoint linear relation $\T$. 
 The proof of the following result (compare Theorem~\ref{thmSARes}) is almost literally the same as the one for Theorem~\ref{thmPencilab}.  
 
 \begin{theorem}
  Suppose that zero belongs to the resolvent set of $\T$.
  The  spectrum of the quadratic operator pencil $\Pe$ coincides with the spectrum of the self-adjoint linear relation $\T$. 
  If some $z\in\C$ belongs to the resolvent set of $\Pe$, then 
  \begin{align}
    \Pe(z)^{-1} g(x) & = \spr{g}{G(x,\cdot\,)^\ast}_{H^1(\R)}, \quad x\in\R, ~g\in H^1(\R),  
  \end{align}
   where the Green's function $G$ is given by  
 \begin{align}
  G(x,s) = \frac{1}{W(\psi_+,\psi_-)} \begin{cases} \psi_+(x) \psi_-(s), & s\leq x, \\ \psi_+(s) \psi_-(x), & s> x,    \end{cases}
 \end{align}
 and $\psi_+$, $\psi_-$ are linearly independent solutions of the homogeneous differential equation~\eqref{eqnDEho} such that $\psi_+$ lies in $H^1(\R)$ near $+\infty$ and $\psi_-$ lies in $H^1(\R)$ near $-\infty$.
  \end{theorem}
 
 The eigenvalues of $\Pe$ consist precisely of those $z\in\C$ for which there is a non-trivial solution $\phi$ of the homogeneous differential equation~\eqref{eqnDEho} which belongs to $H^1(\R)$. 
 Here, note that such a function automatically belongs to $L^2(\R;\dip)$ as well by Theorem~\ref{thmTinv}. 
 In this case, the kernel of $\Pe(z)$ is spanned by the function $\phi$.

\subsection{Singular Weyl--Titchmarsh functions on the whole line}\label{subsecWT}

 We are now going to introduce a singular Weyl--Titchmarsh function associated with $\T$, following essentially \cite{kosate12}  (see also \cite{fu08, fulalu12, gezi06, ka67, ko49}). 
 To this end, we say that some function $\phi: \C\times\R\rightarrow\C$ is {\em a real entire solution} of the homogeneous differential equation~\eqref{eqnDEho} if the following two properties hold: 
  \begin{enumerate}[label=(\roman*), ref=(\roman*), leftmargin=*, widest=iii]   
   \item The function $\phi(z,\redot)$ is a non-trivial solution of the homogeneous differential equation~\eqref{eqnDEho} for every $z\in\C$. 
   \item The functions $\phi(\ledot,x)$ and $\phi'(\ledot,x)$ are real entire for one (and hence by Lemma \ref{lemEE} and Lemma~\ref{lemSolEnt} for all) $x\in\R$. 
  \end{enumerate}
 Now our basic (necessary) prerequisite which will be assumed throughout the remaining part of this section is contained in the following hypothesis. 
 
 \begin{hypothesis}\label{hypRESol}
  There is a real entire solution $\phi$ of the homogeneous differential equation~\eqref{eqnDEho} such that $\phi(z,\cdot\,)$ lies in $\HR$ near $+\infty$ for every $z\in\C$ and 
  \begin{align}\label{eqnRESolnormzero}
   \phi(0,x) = \E^{-\frac{x}{2}}, \quad x\in\R. 
  \end{align}
 \end{hypothesis}
 
 Note that the assumption on the normalization at the origin in~\eqref{eqnRESolnormzero} does not inflict any additional constraints, as it can always be achieved by a simple scaling. 
 
 One way of characterizing Hypothesis~\ref{hypRESol} is contained in the following result (cf.\ \cite[Lemma~3.2]{gezi06}, \cite[\S 5.3]{ka07}, \cite[Lemma~2.2]{kosate12}).  

 \begin{lemma}
  Hypothesis~\ref{hypRESol} holds if and only if the self-adjoint linear relation $\T_{+}$ in $\cH_0(J_{+})$ has purely discrete spectrum for one (and hence all) $c\in\R$. 
 \end{lemma}
 
 \begin{proof}
  If Hypothesis~\ref{hypRESol} holds, then the Weyl--Titchmarsh function $m_{0,+}$ associated with $\T_{0,+}$ is meromorphic in view of~\eqref{eq:wf_pm} for every $c\in\R$. 
  Now Lemma~\ref{lemTrsMdh} guarantees that $\T_{0,+}$ (and hence also $\T_+$) has purely discrete spectrum. 
  
  Conversely, if the linear relation $\T_{+}$ has purely discrete spectrum for some $c\in\R$, then so does $\T_{0,+}$ and the function $m_{0,+}$ is meromorphic by Lemma~\ref{lemTrsMdh}. 
  By the Weierstra\ss\ product theorem, there is a real entire function $h$ which has simple zeros exactly at all (necessarily simple) poles of $m_{0,+}$. 
  Now the function $\phi$ defined by  
  \begin{align*}
   \phi(z,x) = \frac{h(z)}{z} \theta_{0}(z,x) + h(z)m_{0,+}(z) \frac{\phi_{0}(z,x)}{z}, \quad x\in\R,~z\in\C, 
  \end{align*}
  upon recalling~\eqref{eqnpgammapm} and~\eqref{eqntgammapm}, is a real entire solution of the homogeneous differential equation~\eqref{eqnDEho}.  
  Using~\eqref{eq:m_pm}, it is readily verified that $\phi(z,\redot)$ lies in $\HR$ near $+\infty$ for every nonzero $z\in\C$. 
  Furthermore, we infer from Theorem~\ref{thmPencilpm} that 
  \begin{align*}
    \lim_{z\rightarrow0} \frac{\phi_0(z,x)}{z} \left(\theta_0(z,x) + m_{0,+}(z)\phi_0(z,x)\right) = 2 \sinh\left(\frac{x-c}{2}\right) \E^{-\frac{x-c}{2}}, \quad x\in\R,
  \end{align*}
  which shows that $\phi(0,\redot)$ lies in $\HR$ near $+\infty$ as well. 
 \end{proof}
 
 \begin{remark} 
 In this context, let us mention that Hypothesis~\ref{hypRESol} will be satisfied if the measures $\omega$ and $\dip$ have a strong enough decay near $+\infty$. 
 For example, this is immediate under the rather strong assumption that they are not supported near $+\infty$ at all. 
 Then we may simply choose real entire solutions $\theta$, $\phi$ of the homogeneous differential equation~\eqref{eqnDEho} such that   
\begin{align}
 \phi(z,x) & = \E^{-\frac{x}{2}}, & \theta(z,x) & = \E^{\frac{x}{2}},
\end{align}
for all $z\in\C$ and $x$ near $+\infty$. 
More generally, it also suffices to only assume that $\omega$ and $\dip$ have finite total variation near $+\infty$; cf.\ \cite[Theorem~3.1]{IsospecCH}. 
\end{remark}
 
 In order to introduce a singular Weyl--Titchmarsh function associated with $\T$, we furthermore need another, linearly independent real entire solution $\theta$ of the homogeneous differential equation~\eqref{eqnDEho}. 
 The fact that such a solution always exists can be deduced by almost literally following the proof of \cite[Lemma~2.4]{kosate12}. 

 \begin{lemma}\label{lemTheta}
  There is a real entire solution $\theta$ of the homogeneous differential equation~\eqref{eqnDEho} such that $W(\phi,\theta)=1$ and  
  \begin{align}
    \theta(0,x) = \E^{\frac{x}{2}}, \quad x\in\R.
  \end{align}
 \end{lemma}
   
 Given this real entire fundamental system of solutions $\theta$, $\phi$, the complex-valued function $M$ is now defined on $\rho(\T)$ by requiring that the function 
 \begin{align}\label{eqnWTdefwl}
   \theta(z,x) + M(z) \phi(z,x), \quad x\in\R,
 \end{align}
 lies in $\HR$ near $-\infty$ for every $z\in\rho(\T)$. 
 In view of Corollary~\ref{corTrespsipm}, this function is well-defined and henceforth called {\em the singular Weyl--Titchmarsh function} associated with $\T$. 
 Note that due to the normalization of the real entire solutions $\theta$ and $\phi$ at zero, we have $M(0) = 0$, as long as zero belongs to the resolvent set of $\T$. 

If $\psi(z,\redot)$ is a non-trivial solution of the homogeneous differential equation~\eqref{eqnDEho} which lies in $\HR$ near $-\infty$, then for any $c\in\R$, the function $M$ is given by 
 \begin{align}\label{eqn_mpsi}
   M(z) = \frac{W(\psi,\theta)(z)}{W(\phi,\psi)(z)} = \frac{\psi(z,c)\theta'(z,c) - \psi'(z,c)\theta(z,c)}{\phi(z,c)\psi'(z,c) - \phi'(z,c)\psi(z,c)}, \quad z\in\rho(\T).
 \end{align} 
 Note that in this case, $\psi(z,\redot)$ is a scalar multiple of the function in~\eqref{eqnWTdefwl}. 

\begin{lemma}\label{lemManal}
 The function $M$ is analytic on $\rho(\T)$ with 
 \begin{align}\label{eqnMconj}
  M(z)^\ast = M(z^\ast), \quad z\in\rho(\T).
 \end{align}
\end{lemma}

\begin{proof}
 From Theorem~\ref{thmSARes} we get for every $c\in\R$, upon recalling~\eqref{eqnPEdeltac},
 \begin{align}\label{eqnTresMwl} 
   M(z) \phi(z,c)^2 =   z\, \spr{(\T-z)^{-1} \delta_c}{\delta_c}_{\HR} & - \theta(z,c)\phi(z,c) + 1, \quad z\in\rho(\T). 
 \end{align}
 But this shows that $M$ is analytic on the resolvent set of $\T$ as well as~\eqref{eqnMconj}, since for each $z\in\C$ there is some $c\in\R$ such that $\phi(z,c)\not=0$. 
 \end{proof}

 Note that the subtle difference between~\eqref{eqnTresMwl} and~\eqref{eq:mRescon}, namely the additional $z$ term on the right-hand side of~\eqref{eq:mRescon}, comes from the altered normalization $W(\phi,\theta)=1$ (instead of $W(\phi,\theta) = z$) of the fundamental system of solutions. 

 Actually, we can improve on Lemma~\ref{lemManal} and show that the spectrum of $\T$ can be read off from the singular Weyl--Titchmarsh function $M$ (cf.\ \cite[Corollary~3.5]{kosate12}). 
 
 \begin{lemma}
  The resolvent set of the self-adjoint linear relation $\T$ coincides with the maximal domain of holomorphy of the singular Weyl--Titchmarsh function $M$. 
 \end{lemma}

 \begin{proof}
  The claim follows in much the same manner as in Lemma~\ref{lemTrsMdh}. 
 \end{proof}

\begin{remark}\label{remSecFS}
 Concluding, let us mention that real entire solutions as in Hypothesis~\ref{hypRESol} and Lemma~\ref{lemTheta} are not unique. 
 In fact, any other such fundamental system is given by 
 \begin{align}
  \tilde{\theta}(z,x) & = \E^{-g(z)} \theta(z,x) - f(z) \phi(z,x), & \tilde{\phi}(z,x) = \E^{g(z)} \phi(z,x)
 \end{align}
 for some real entire functions $f$ and $g$ with $f(0) = g(0)=0$.  
 The corresponding singular Weyl--Titchmarsh functions are then simply related via
 \begin{align}
  \tilde{M}(z) = \E^{-2g(z)} M(z) + \E^{-g(z)} f(z), \quad z\in\rho(\T).
 \end{align}
 In particular, the maximal domain of holomorphy or the structure of poles and singularities of the singular Weyl--Titchmarsh functions do not change.
\end{remark}

\subsection{The spectral transformation}\label{subsST}

 As the next step, we will now show that it is possible to associate a spectral measure with the singular Weyl--Titchmarsh function $M$ introduced in the preceding subsection. 
 To this end, recall that for all $f$, $g\in\HR$ there is a unique complex Borel measure $E_{f,g}$ on $\R$ such that
\begin{align}\label{eqnLDstietrans}
  \spr{ (\T-z)^{-1} f}{g}_{\HR} = \int_\R \frac{1}{\lambda-z} dE_{f,g}(\lambda), \quad z\in\rho(\T).
\end{align}
 
\begin{lemma}\label{lemSM}
There is a unique non-negative Borel measure $\mu$ on $\R$ such that 
\begin{align}\label{eqnSMdens}
 E_{\delta_a,\delta_b}(B) = \int_{B} \phi(\lambda,a)  \phi(\lambda,b) d\mu(\lambda)
\end{align}
for all $a$, $b\in\R$ and every Borel set $B\subseteq\R$.
\end{lemma}

\begin{proof}
 As in the proof of Lemma~\ref{lemManal}, one obtains for all $a$, $b\in\R$
\begin{align*}
 \spr{(\T-z)^{-1} \delta_a}{\delta_b}_{\HR} = \frac{M(z)}{z} \phi(z,a)\phi(z,b) + H_{a,b}(z), \quad z\in\rho(\T),
\end{align*}
where $H_{a,b}$ is a real entire function. 
With the Borel measure $\mu$ defined by 
\begin{align}\label{eqndefmu}
 \mu([\lambda_1,\lambda_2)) = \lim_{\delta\downarrow 0} \lim_{\varepsilon\downarrow 0} \frac{1}{\pi} \int_{\lambda_1-\delta}^{\lambda_2-\delta}  \im\left(\frac{M(\lambda+\I\varepsilon)}{\lambda+\I\varepsilon}\right) d\lambda, 
\end{align}
for $\lambda_1$, $\lambda_2\in\R$ with $\lambda_1<\lambda_2$, the claim follows along the lines of~\cite[Lemma~3.3]{kosate12}. 
\end{proof}

 The measure $\mu$  introduced in Lemma~\ref{lemSM} and defined by~\eqref{eqndefmu} will turn out to be a central object in spectral theory for $\T$.  
 As a next step, we define the transform  
 \begin{align}\begin{split}\label{eqnFhat}
   \hat{f}(z) = \frac{1}{4} \int_\R \phi(z,x) f_1(x) & dx + \int_\R \phi'(z,x) f_1'(x)dx \\ 
                           & + \int_\R z \phi(z,x)f_2(x) d\dip(x), \quad z\in\C,  
 \end{split}\end{align}
 for any function with compact support $f\in\cH_{\mathrm{c}}(\R)$.

\begin{lemma}\label{lemSMst}
 Given compactly supported functions $f$, $g\in\cH_{\mathrm{c}}(\R)$, we have 
\begin{align}
 E_{f,g}(B) = \int_{B} \hat{f}(\lambda)  \hat{g}(\lambda)^\ast d\mu(\lambda)
\end{align}
for every Borel set $B\subseteq\R$.
\end{lemma}

\begin{proof}
Due to the polarization identity, it suffices to prove the claim in the case when $f=g$. 
Using Theorem~\ref{thmSARes}, a lengthy but straightforward calculation gives  
\begin{align*}
 \spr{(\T-z)^{-1} f}{f}_{\HR} = \frac{M(z)}{z} \hat{f}(z) \hat{f}(z^\ast)^\ast + H_{f}(z), \quad z\in\rho(\T),
\end{align*}
for some real entire function $H_{f}$ (also note that the transform $\hat{f}$ is entire). 
Here, one should note that $\theta$ and $\phi$ are locally uniformly bounded on $\C\times\R$ (which can be deduced from the Gronwall lemma \cite[Lemma~1.3]{be89}, \cite[Lemma~A.1]{MeasureSL}). 
Now the claim can again be verified in much the same manner as~\cite[Lemma~3.3]{kosate12}. 
\end{proof}

In particular, the preceding lemma shows that the mapping $f\mapsto \hat{f}$ is a partial isometry from a dense subspace of $\HR$ into $\Lmu$.
More precisely, for every function $f\in\HR$ with compact support we obtain 
\begin{align}
 \| \hat{f}\|^2_{\Lmu} = \int_\R \hat{f}(\lambda) \hat{f}(\lambda)^\ast d\mu(\lambda) = \int_\R dE_{f,f} =  \|\mathrm{P} f\|_{\HR}^2,
\end{align}
where $\mathrm{P}$ is the orthogonal projection onto the closure $\D$ of $\dom{\T}$ (note that the domain of $\T$ can be non-dense indeed since the multi-valued part of $\T$ can be non-trivial). 
Consequently, we may extend this mapping uniquely to a partial isometry $\F$ from $\HR$ into $\Lmu$ with initial subspace $\D$. 
Of course, the result of Lemma~\ref{lemSMst} now immediately extends to all functions $f$, $g\in\HR$, that is, 
\begin{align}\label{eqnSMF}
 E_{f,g}(B) =  \int_B \F f(\lambda) \F g(\lambda)^\ast d\mu(\lambda)
\end{align}
for every Borel set $B\subseteq\R$. 
The obvious similarity of the assertions in Lemma~\ref{lemSM} and in Lemma~\ref{lemSMst} becomes clear in view of the following result. 

\begin{proposition}\label{propTransPE}
 For every $c\in\R$ we have 
 \begin{align}
  \F\delta_c(\lambda) = \phi(\lambda,c)
 \end{align} 
 for almost all $\lambda\in\R$ with respect to $\mu$.
\end{proposition}

\begin{proof}
Given $c\in\R$ and a real-valued $f\in\cH_{\mathrm{c}}(\R)$, we infer from Theorem~\ref{thmSARes} that    
\begin{align*}
 \spr{(\T-z)^{-1} f}{\delta_c}_{\HR} = \frac{M(z)}{z} \hat{f}(z) \phi(z,c) + H_{f,c}(z), \quad z\in\rho(\T),
\end{align*}
for some real entire function $H_{f,c}$. 
Once more, we conclude as in~\cite[Lemma~3.3]{kosate12} 
\begin{align*}
 E_{f,\delta_c}(B) = \int_B \hat{f}(\lambda) \phi(\lambda,c) d\mu(\lambda) 
\end{align*} 
for every Borel set $B\subseteq\R$, which again extends to all functions $f\in\HR$ in a straightforward way. 
In view of~\eqref{eqnSMF} and Lemma~\ref{lemSM},  this finally yields  
\begin{align*}
    \spr{\F\delta_c}{\phi(\ledot,c)}_{\Lmu} & = E_{\delta_c,\delta_c}(\R) = \|\F\delta_c\|_{\Lmu}^2 = \|\phi(\ledot,c)\|_{\Lmu}^2,
\end{align*} 
proving the claim. 
\end{proof}

It will turn out that the transformation $\F$ maps the self-adjoint linear relation $\T$ to multiplication with the independent variable in $\Lmu$. 
Before we get to prove this, we first derive a few more properties of this transformation. 

\begin{lemma}\label{lemFadjoint}
 The adjoint of the operator $\F$ is given by 
 \begin{align}
  \F^\ast g(x) = \lim_{r\rightarrow\infty} \int_{-r}^r \begin{pmatrix} 1 \\ \lambda \end{pmatrix} \phi(\lambda,x) g(\lambda) d\mu(\lambda), \quad x\in\R,~g\in\Lmu, 
 \end{align} 
 where the limit has to be understood as a limit in $\HR$. 
\end{lemma}

\begin{proof}
 For every function $g\in\Lmu$ with compact support, we set
 \begin{align*}
  \check{g}(x) = \int_\R \begin{pmatrix} 1 \\ \lambda \end{pmatrix} \phi(\lambda,x) g(\lambda) d\mu(\lambda), \quad x\in\R, 
 \end{align*}
 and note that $\check{g}_1$ belongs to $H^1(\R)$ since 
 \begin{align*}
   \check{g}_1(x) & 
                            = \spr{g}{\F\delta_x}_{\Lmu}  = \spr{\F^\ast g}{\delta_x}_{\HR}   
                            =  (\F^\ast g)_1(x), \quad x\in\R.
 \end{align*} 
 Now for arbitrary $a$, $b\in\R$ with $a<b$ we obtain upon interchanging integrals 
\begin{align*}
 L_{a,b}^{2} & = \int_a^b \left|\check{g}_2(x)\right|^2 d\dip(x)  
            =  \int_a^b \check{g}_2(x) \int_\R \lambda \phi(\lambda,x) g(\lambda)^\ast d\mu(\lambda)\, d\dip(x)  \\
          & = \int_\R g(\lambda)^\ast \int_a^b \lambda\phi(\lambda,x) \check{g}_2(x) d\dip(x) d\mu(\lambda) 
           = \int_\R  g(\lambda)^\ast \F\begin{pmatrix} 0 \\ \indik_{[a,b)} \check{g}_2 \end{pmatrix}(\lambda) \,  d\mu(\lambda) \\         
          & \leq \left\| g\right\|_{\Lmu} \left\| \F\begin{pmatrix} 0 \\ \indik_{[a,b)} \check{g}_2 \end{pmatrix}\right\|_{\Lmu}  
          \leq \left\| g\right\|_{\Lmu} L_{a,b},         
\end{align*}
 which shows that $\check{g}$ belongs to $\HR$. 
 Now if $f\in\HR$ is such that $f_1=0$ and $f_2$ has compact support, then upon interchanging integrals one sees 
 \begin{align*}
  \spr{\check{g}_2}{f_2}_{L^2(\R;\dip)} = \spr{g}{\hat{f}}_{\Lmu} & = \spr{\F^\ast g}{f}_{\HR} = \spr{(\F^\ast g)_2}{f_2}_{L^2(\R;\dip)},  
 \end{align*}
 implying $\check{g} = \F^\ast g$ and hence the claim. 
 \end{proof}

\begin{lemma}\label{lemSTonto}
 The mapping $\F$ is onto with (in general multi-valued) inverse 
 \begin{align}
  \F^{-1} = \F^\ast \oplus \left(\lbrace0\rbrace \times \mul{\T}\right). 
 \end{align}
\end{lemma}

 \begin{proof}
  Let $\lambda_0\in\R$ and choose some $x\in\R$ such that $\phi(\lambda_0,x)\not=0$. 
  Then for every small enough interval $J\subseteq\R$ around $\lambda_0$, the function
 \begin{align*}
  G(\lambda) = \begin{cases}
                \phi(\lambda,x)^{-1}, & \lambda\in J, \\
                0,                    & \lambda\in\R \backslash J,
               \end{cases}
 \end{align*}
 is bounded. 
 By a variant of the spectral theorem, we infer that there is a $g\in\HR$ such that 
 $\F g(\lambda) = G(\lambda)\F \delta_x(\lambda) = \indik_J$ for almost all $\lambda\in\R$ with respect to $\mu$. 
 Therefore, the range of $\F$ contains all characteristic functions of bounded intervals. 
 But this shows that $\F$ is onto since the range of a partial isometry is always closed. 
 
 To verify the remaining claim, it suffices to note that $\F\F^\ast$ is the identity operator in $\Lmu$ and $\F^\ast\F$ is the orthogonal projection onto $\D=\mul{\T}^\bot$ in $\HR$. 
\end{proof}

 In the following, we will denote with $\M_{\mathrm{id}}$ the maximally defined operator of multiplication with the independent variable in $\Lmu$.

\begin{theorem}\label{th:TsimM}
 The transformation $\F$ maps the self-adjoint linear relation $\T$ to multiplication with the independent variable in $\Lmu$.  
 More precisely,   
 \begin{align}
  \F\,\T\, \F^\ast = \M_{\mathrm{id}}.
 \end{align} 
\end{theorem}

\begin{proof}
 First of all, we infer from~\eqref{eqnSMF} that for every $g\in\Lmu$ one has  
 \begin{align*}
  g\in\dom{\M_{\mathrm{id}}} 
              & \quad\Leftrightarrow\quad \F^\ast g\in\dom{\T} \quad\Leftrightarrow\quad  g\in\dom{\F\,\T\,\F^\ast}.
 \end{align*}
 In this case, equation~\eqref{eqnSMF} and~\cite[Lemma~B.4]{MeasureSL} show  that 
 \begin{align*}
   \spr{\M_{\mathrm{id}} g}{h}_{\Lmu} & = \int_\R \lambda\, g(\lambda) h(\lambda)^\ast d\mu(\lambda) = \int_\R \lambda\, dE_{\F^\ast g,\F^\ast h}(\lambda) \\
                                                            & = \spr{f}{\F^\ast h}_{\HR}  = \spr{\F f}{h}_{\Lmu}, \quad h\in\Lmu, 
 \end{align*}
 whenever $(\F^\ast g,f)\in\T$, which yields the claim. 
 \end{proof}

 Note that Theorem~\ref{th:TsimM} establishes a connection between the spectral properties of $\T$ and $\M_{\mathrm{id}}$. 
 In particular, the spectrum of $\T$ coincides with the support of the measure $\mu$. 
 The mass of $\mu$ at an eigenvalue may be given explicitly in terms of $\phi$.

\begin{corollary}
 If $\lambda_0\in\sigma(\T)$ is an eigenvalue of $\T$, then 
 \begin{align}\begin{split}
  \mu(\lbrace \lambda_0\rbrace)^{-1} & 
     = \frac{1}{4} \int_\R |\phi(\lambda_0,x)|^2 dx + \int_\R |\phi'(\lambda_0,x)|^2dx + \int_\R |\lambda_0\phi(\lambda_0,x)|^2 d\dip(x).  
 \end{split}\end{align}
\end{corollary}

\begin{proof}
 It follows from Theorem~\ref{th:TsimM} that the transform of the function $\Phi\in\D$,  
 \begin{align*}
  \Phi(x) = \begin{pmatrix} 1 \\ \lambda_0 \end{pmatrix} \phi(\lambda_0,x), \quad x\in\R, 
 \end{align*} 
 is given by $\F\Phi(\lambda_0) \indik_{\lbrace\lambda_0\rbrace}$, where $\F\Phi(\lambda_0) = \|\Phi\|_{\HR}^2$ in view of~\eqref{eqnFhat}.  
 Now the claim follows from the fact that $\F$ is a partial isometry. 
\end{proof}

The measure $\mu$ is uniquely determined by the property that the mapping $f\mapsto\hat{f}$ uniquely extends to a partial isometry onto $\Lmu$, which maps $\T$ onto multiplication with the independent variable in $\Lmu$. 
For this reason, the measure $\mu$ is referred to as {\em the spectral measure} of $\T$ associated with the real entire solution $\phi$.

\begin{remark}
 Given any other real entire fundamental system as in Remark~\ref{remSecFS}, the corresponding spectral measures are related via  
 \begin{align}
  \tilde{\mu}(B) = \int_B \E^{-2g(\lambda)} d\mu(\lambda), 
 \end{align}
for every Borel set $B\subseteq\R$. 
 In particular, the measures are mutually absolutely continuous and the associated spectral transforms just differ by a simple rescaling. 
\end{remark}

 One of the most important properties of Herglotz--Nevanlinna functions is the existence of an integral representation.
 By this means, it is possible to relate (classical) Weyl--Titchmarsh functions to their associated spectral measures. 
 Such an integral representations also exists for the singular Weyl--Titchmarsh function $M$, relating it to the spectral measure $\mu$. 
 We omit to state and prove this result here, which can be done along the lines of \cite[Section~4]{kosate12}.

\section{Inverse spectral theory}\label{secIST}

In this final section, we are going to provide some basic inverse uniqueness theorems for our spectral problem.
The proofs of these results rely on de Branges' subspace ordering theorem for certain Hilbert spaces of entire functions. 
For an exposition of de Branges' theory, we refer to de Branges' book \cite{dB68} as well as to \cite{dy70, dymc76, re02}, in which particular emphasis is placed on its applications to spectral theory for Sturm--Liouville operators.

\subsection{Paley--Wiener spaces}\label{subsecdB}

As a preparatory step, we will first discuss the main properties of a certain family of de Branges spaces associated with our spectral problem under the additional assumption of Hypothesis~\ref{hypRESol}.
To this end, we fix some $c\in\R$ and introduce the entire function
\begin{align}
 E(z,c) = z \phi(z,c) - \I \phi'(z,c), \quad z\in\C,
\end{align}
as well as the function
\begin{align}\label{eqnEdB}
   K(\zeta,z,c) = \frac{E(z,c)E(\zeta,c)^\ast - E(\zeta^\ast,c)E(z^\ast,c)^\ast}{2\I (\zeta^\ast-z)}, \quad \zeta,\, z\in\C.
\end{align}
 Employing the Lagrange identity in Proposition~\ref{propLagrange} as well as Lemma~\ref{lem:4.02} yields  
\begin{align}\begin{split}\label{eqnKint}
 K(\zeta,z,c) =  \frac{1}{4} \int_{c}^\infty  \phi(z,x) \phi(\zeta^\ast,x) & dx + \int_{c}^\infty  \phi'(z,x)  \phi'(\zeta^\ast,x) dx \\
                          &  +  \int_{c}^{\infty} z\phi(z,x)  \zeta^\ast \phi(\zeta^\ast,x) d\dip(x), \quad \zeta,\, z\in\C. 
\end{split}\end{align}
Upon choosing $\zeta=z$, this equation shows that $E(\ledot,c)$ is a de Branges function, that is, $|E(z,c)| > |E(z^\ast,c)|$ for all $z$ in the open upper complex half-plane $\H$. 

 The associated de Branges space $\cB(c)$ (see, for example, \cite[Section~19]{dB68} or \cite[Section~2]{re02}) consists of all entire functions $F$ such that $F(\ledot)/E(\ledot,c)$ and 
 $F(\ledot^\ast)^\ast/E(\ledot,c)$ belong to the Hardy space $H^2(\H)$ for the upper half-plane $\H$.
 Equipped with the scalar product
 \begin{align}
  \spr{F}{G}_{\cB(c)} = \frac{1}{\pi} \int_\R \frac{F(\lambda) G(\lambda)^\ast}{|E(\lambda,c)|^2} d\lambda, \quad F,\, G\in \cB(c), 
 \end{align}
 the space $\cB(c)$ turns into a reproducing kernel Hilbert space with 
 \begin{align}\label{eq:6.05}
  F(\zeta) = \spr{F}{K(\zeta,\cdot\,,c)}_{\cB(c)}, \quad F\in\cB(c),
 \end{align}
 for every $\zeta\in\C$, which can be seen readily from \cite[Theorem~19]{dB68} and \eqref{eqnEdB}.  

In order to reveal the de Branges space $\cB(c)$ as a Paley--Wiener space corresponding to the generalized Fourier transform $\F$, we define the transform 
 \begin{align}\label{eqnfUc}\begin{split}
 \U_c f(z) = \frac{1}{4} \int_{c}^\infty \phi(z,x) f_1(x) & dx + \int_{c}^\infty \phi'(z,x) f_1'(x) dx \\
         & +  \int_{c}^\infty z\phi(z,x) f_2(x) d\dip(x), \quad z\in\C, 
\end{split}\end{align}
for every function $f\in\cH([c,\infty))$.  
Moreover, we let $\D(c)$ be the smallest closed subspace of $\cH([c,\infty))$ which contains all the functions $\Phi_c(z,\redot)$, $z\in\C$, where  
\begin{align}
 \Phi_c(z,x) & = \begin{pmatrix} 1 \\ z \end{pmatrix} \phi(z,x), \quad x\in[c,\infty). 
\end{align}

\begin{lemma}\label{lemdBTrans}
 For each $c\in\R$, the transformation $\U_c$ is a surjective partial isometry from $\cH([c,\infty))$ onto $\cB(c)$ with initial subspace $\D(c)$. 
\end{lemma}

\begin{proof}
 For every $\zeta\in\C$, the transform of the function $\Phi_c(\zeta,\cdot\,)$ is simply given by $K(\zeta^\ast,\redot,c)$, which obviously belongs to the de Branges space $\cB(c)$. 
 Furthermore, for any given $\zeta_1$, $\zeta_2\in\C$ we have
 \begin{align*}
  \spr{\Phi_c(\zeta_1,\cdot\,)}{\Phi_c(\zeta_2,\cdot\,)}_{\cH([c,\infty))} & = K(\zeta_1^\ast,\zeta_2^\ast,c) 
    = \spr{K(\zeta_1^\ast,\cdot\,,c)}{K(\zeta_2^\ast,\cdot\,,c)}_{\cB(c)} \\ 
   & = \spr{\U_c \Phi_c(\zeta_1,\cdot\,)}{\U_c \Phi_c(\zeta_2,\cdot\,)}_{\cB(c)}.
 \end{align*}
 Since the linear span of the functions $K(\zeta,\redot,c)$, $\zeta\in\C$ is dense in $\cB(c)$, there is a surjective partial isometry $\mathcal{V}_c$ from $\cH([c,\infty))$ onto $\cB(c)$ with initial subspace $\D(c)$, which coincides with $\U_c$ on the functions $\Phi_c(\zeta,\cdot\,)$, $\zeta\in\C$. 
 In order to identify $\mathcal{V}_c$ with $\U_c$, first note that the functionals $f\mapsto\mathcal{V}_c f(z)$ and $f\mapsto\U_c f(z)$ are continuous on $\D(c)$ for every $z\in\C$, which shows that $\mathcal{V}_c$ coincides with $\U_c$ on $\D(c)$. 
 Secondly, if $f\in\cH([c,\infty))$ is orthogonal to $\D(c)$, then it follows readily from the definition of $\D(c)$ that $\U_c f(z) = 0$ for every $z\in\C$, which shows that $\mathcal{V}_c$ coincides with $\U_c$.
\end{proof}

The closed linear subspace $\cB_0(c)$ of functions in $\cB(c)$, which vanish at the origin, 
\begin{align}
\cB_0(c) = \lbrace F\in \cB(c) \,|\, F(0)=0 \rbrace,
\end{align}
 will take a particular role, as it is exactly the image of $\cH_0([c,\infty))$ under the transformation $\U_c$. 
In fact, an integration by parts shows that 
\begin{align}\label{eqndBTranszero}
 \U_c f(0) 
                    = \frac{1}{2} \E^{-\frac{c}{2}} f_1(c), \quad f\in \cH([c,\infty)).
\end{align}
The orthogonal complement of $\cB_0(c)$ corresponds to the orthogonal complement of $\cH_0([c,\infty))$, which is spanned by the function $\Phi_c(0,\redot)$. 
 More precisely, the transform of this function is given by 
\begin{align}\label{eqndBTransOrth}
 \U_c \Phi_c(0,\redot)(z) =  K(0,z,c) 
                                                = \frac{1}{2} \E^{-\frac{c}{2}} \phi(z,c), \quad z\in\C.
\end{align}

A few crucial properties of the de Branges spaces $\cB(c)$ only hold if $c$ belongs to the set $\Sigma = \supp(|\omega|+\dip)$, that is, the topological support of the measure $|\omega| + \dip$. 
From the characterization of the multi-valued part in~\eqref{eq:multpart}, we immediately see that the closure $\D$ of the domain of $\T$ is given by
\begin{align}\label{eqnCDomT}
  \D = \overline{\dom{\T}} = \mul{\T}^\bot =  \overline{\mathrm{span}\lbrace \delta_c \,|\, c\in\Sigma \rbrace}\times L^2(\R;\dip).
\end{align}
 Let us mention that the first components of functions in $\D$ are uniquely determined by their values on $\Sigma$. 
 In fact, if $f$, $g\in \D$ such that $f_1(c)=g_1(c)$ for all $c\in\Sigma$, then $f_1 - g_1$ is orthogonal to $\mathrm{span}\lbrace \delta_c \,|\, c\in\Sigma \rbrace$ and hence $f_1=g_1$ due to~\eqref{eqnCDomT}.

We are now ready to state and prove the main result of the present subsection; an embedding theorem for the de Branges spaces $\cB(c)$ into the space $\Lmu$, where $\mu$ is the spectral measure of $\T$ associated with the real entire solution $\phi$. 

\begin{theorem}\label{thmdBemb}
 For each $c\in\Sigma$, the de Branges space $\cB(c)$ is homeomorphically embedded in $\Lmu$ with 
 \begin{align}\label{eqndBemb}
  \spr{F}{G}_{\Lmu} = \spr{F}{G}_{\cB(c)} - \E^{c} F(0) G(0)^\ast, \quad F,\, G\in \cB(c).
 \end{align}
\end{theorem}

\begin{proof}
First of all note that for every $z\in\C$ and $h\in\mul{\T}$ we obtain 
\begin{align}\begin{split}\label{eqnLDpartinmul}
  \frac{1}{4} \int_{c}^\infty \phi(z,x) h_1(x)^\ast & dx + \int_c^\infty \phi'(z,x) h_1'(x)^\ast dx \\ 
        & + \int_{c}^\infty z\phi(z,x) h_2(x)^\ast d\dip(x) = \lim_{x\rightarrow \infty} \phi'(z,x) h_1(x)^\ast = 0.
\end{split}\end{align}
Here we used the fact that $h_2=0$, that $h_1$ vanishes almost everywhere with respect to $|\omega|+\dip$ (in particular, also note that $h_1(c)=0$), as well as Lemma~\ref{lem:4.02}. 

Now pick some arbitrary functions $f$, $g\in\mathrm{span}\lbrace \Phi_c(z,\redot) \,|\, z\in\C\rbrace$ and set  
\begin{align*}
 f_\triangleright & = f_1(c) \E^{\frac{c}{2}} \Phi_c(0,\redot), & f_\circ & = f - f_\triangleright,
\end{align*}
as well as similarly for the function $g$.
By setting them equal to zero outside of $[c,\infty)$, we may extend $f_\circ$, $g_\circ$ to functions $\bar{f}_\circ$, $\bar{g}_\circ\in\HR$. 
These extensions even belong to $\D$, as~\eqref{eqnLDpartinmul} shows that they are orthogonal to $\mul{\T}$.
Hence we get 
\begin{align*}\begin{split}
 \spr{\U_c f_\circ}{\U_c g_\circ}_{\Lmu}  = \spr{\F\bar{f}_\circ}{\F\bar{g}_\circ}_{\Lmu} =& \spr{\bar{f}_\circ}{\bar{g}_\circ}_{\HR} \\ = \spr{f_\circ}{g_\circ}_{\cH([c,\infty))} &  = \spr{\U_c f_\circ}{\U_c g_\circ}_{\cB(c)},
\end{split}\end{align*}
where we used that $\F$ is given by~\eqref{eqnFhat} and an isometry on $\D$, as well as Lemma~\ref{lemdBTrans}. 
Moreover, from~\eqref{eqndBTransOrth}, Proposition~\ref{propTransPE} (also note that $\delta_c\in\D$ and $\|\delta_c\|_{\HR} = 1$) and finally also~\eqref{eqndBTranszero} we get 
\begin{align*}
 \spr{\U_c f_\triangleright}{\U_c g_\triangleright}_{\Lmu} & = \frac{1}{4} f_1(c) g_1(c)^\ast \int_\R |\phi(\lambda,c)|^2 d\mu(\lambda) = \E^{c}\, \U_c f(0)\, \U_c g(0)^\ast.
\end{align*}
Furthermore, in a similar way one arrives at 
 \begin{align*}
  \spr{\U_c f_\circ}{\U_c g_\triangleright}_{\Lmu} & = \frac{1}{2} g_1(c)^\ast \int_\R \F \bar{f}_\circ(\lambda) \phi(\lambda,c) d\mu(\lambda) = \frac{1}{2} g_1(c)^\ast \spr{\bar{f}_\circ}{\delta_c}_{\HR} = 0,
 \end{align*}
 that is, the entire function $\U_c g_\triangleright$ is orthogonal to $\U_c f_\circ$ not only in $\cB(c)$ but also in the space $\Lmu$.
Using all these properties, we finally end up with
 \begin{align*}
  \spr{\U_c f}{\U_c g}_{\Lmu} & = \spr{\U_c f_\circ}{\U_c g_\circ}_{\Lmu} + \spr{\U_c f_\triangleright}{\U_c g_\triangleright}_{\Lmu}  \\ 
                &  = \spr{\U_c f}{\U_c g}_{\cB(c)} - \E^{c}\, \U_c f(0)\, \U_c g(0)^\ast,            
 \end{align*}
 where we also employed the simple identity
 \begin{align*}
  \spr{\U_c f_\triangleright}{\U_c g_\triangleright}_{\cB(c)} = \E^{c} f_1(c) g_1(c)^\ast K(0,0,c) = 2 \E^{c}\, \U_c f(0)\, \U_c g(0)^\ast.
 \end{align*}
 This guarantees that~\eqref{eqndBemb} holds for all $F$, $G$ in a dense subspace of $\cB(c)$. 
 Now by approximation, one shows that all $F$, $G\in\cB(c)$ belong to $\Lmu$ such that~\eqref{eqndBemb} holds. 
 Finally, the embedding is homeomorphic since the expression on the right-hand side of~\eqref{eqndBemb} gives rise to a norm on $\cB(c)$ which is equivalent to $\|\cdot\|_{\cB(c)}$. 
 \end{proof}

In particular, note that under the assumption of Theorem~\ref{thmdBemb} the subspace $\cB_0(c)$ is isometrically embedded in $\Lmu$.
Moreover, the orthogonal complement of $\cB_0(c)$ in $\cB(c)$ is also orthogonal to $\cB_0(c)$ in $\Lmu$.
For functions $F\in\cB(c)$ which are orthogonal to $\cB_0(c)$ we have 
\begin{align}\label{eqndBnormOC}
  \|F\|_{\cB(c)} = \sqrt{2} \E^{\frac{c}{2}} |F(0)| =  \sqrt{2} \| F\|_{\Lmu}.
\end{align} 
The difference between $\cB_0(c)$ and its orthogonal complement stems from the fact that the space $\cH_0([c,\infty))$ corresponding to $\cB_0(c)$ is isometrically embedded in $\HR$, whereas the whole space $\cH([c,\infty))$ is not.

In the remaining part of this subsection, we will provide several further properties of the de Branges spaces $\cB(c)$ which will turn out to be useful for the proof of the inverse uniqueness theorem in the following subsection. 
However, before we do this, we first introduce for every $c\in\R$ the entire function 
\begin{align}\label{eqnEcplus}
 E(z,c+) = z \phi(z,c) - \I \phi'(z,c+), \quad z\in\C.
\end{align}
 Similarly as above, one shows that this function is a de Branges function giving rise to a de Branges space $\cB(c+)$ with reproducing kernel given by 
\begin{align}\label{eqnKcplus}
 K(\zeta,z,c+) = K(\zeta,z,c) - \dip(\lbrace c\rbrace)\, z\phi(z,c)\, \zeta^\ast\phi(\zeta,c)^\ast, \quad \zeta,\,z\in\C. 
\end{align}
 Furthermore, we also introduce the closed subspace 
 \begin{align}
  \cB_0(c+) = \lbrace F\in \cB(c+) \,|\, F(0)=0 \rbrace.
 \end{align}
 Along the lines of the proof of Theorem~\ref{thmdBemb}, one can prove that for each $c\in\Sigma$, the de Branges space $\cB(c+)$ is homeomorphically embedded in $\Lmu$ as well with 
  \begin{align}\label{eqndBembplus}
  \spr{F}{G}_{\Lmu} = \spr{F}{G}_{\cB(c+)} - \E^{c} F(0) G(0)^\ast, \quad F,\, G\in \cB(c+).
 \end{align}

\begin{proposition}\label{propdBprop}
The de Branges spaces $\cB(c)$ have the following properties:
\begin{enumerate}[label=(\roman*), ref=(\roman*), leftmargin=*, widest=iiii] 
\item \label{itemdBplus}
 For every $c\in\R$ one has $\cB(c)\supseteq\cB(c+)$. 
 The inclusion is strict if and only if $\dip$ has mass in $c$. 
 In this case, $\cB(c+)$ has codimension one in $\cB(c)$ with  
 \begin{align}\label{eqnBcBcplus}
  \cB(c) =  \cB(c+) \oplus \mathrm{span}\lbrace z\mapsto z\phi(z,c) \rbrace. 
 \end{align}
\item \label{itemdBincl} 
 For every $c_1$, $c_2\in\R$ with $c_2<c_1$ one has $\cB(c_2+)\supseteq \cB(c_1)$.
 The inclusion is strict if and only if the intersection $(c_2,c_1]\cap\Sigma$ is not empty. 
 In this case, if $(c_2,c_1)\cap\Sigma$ is empty, then $\cB(c_1)$ has codimension one in $\cB(c_2+)$ with
 \begin{align}\label{eqnBcplusgap}
  \cB(c_2+) = \cB(c_1) \,\dot{+}\, \mathrm{span}\lbrace z\mapsto \phi'(z,c_1)\rbrace.
 \end{align}
 \item \label{itemdBcont}
 For every $c$, $a_n$, $b_n\in\Sigma$ with $a_n\uparrow c$ and $b_n\downarrow c$ as $n\rightarrow\infty$ one has 
\begin{align}\label{eqndBcont}
  \overline{\bigcup_{n\in\N} \cB(b_n)} & = \cB(c+), & \bigcap_{n\in\N} \cB(a_n) & = \cB(c).
\end{align}
\item  \label{itemdBempty}
 Unless $\Sigma$ is empty, one has 
\begin{align}
 \label{eqndBempty} \bigcap_{c\,\in\,\Sigma} \cB(c+) & = \begin{cases} \lbrace 0\rbrace, & \sup\Sigma = \infty, \\ \cB(\sup\Sigma+), & \sup\Sigma<\infty. \end{cases} 
\end{align}
 In the latter case, the space $\cB(\sup\Sigma+)$ is one dimensional. 
\item \label{itemdBdense}
 Unless $\Sigma$ is empty, one has 
\begin{align}
 \label{eqndBdense} \overline{\bigcup_{c\,\in\,\Sigma} \cB(c)} & = \begin{cases} \Lmu, & \inf\Sigma=-\infty, \\ \cB(\inf\Sigma), & \inf\Sigma>-\infty. \end{cases} 
\end{align}
 In the latter case, the space $\cB(\inf\Sigma)$ has codimension one in $\Lmu$. 
 \end{enumerate}
\end{proposition}

\begin{proof}
 \ref{itemdBplus}
  The inclusion follows from~\eqref{eqnKcplus} and \cite[Theorems~I.3 and I.6]{ar50}. 
 If $\dip$ does not have mass in $c$, then the reproducing kernels $K(\ledot,\redot,c)$ and $K(\ledot,\redot,c+)$ coincide and thus so do the spaces $\cB(c)$ and $\cB(c+)$. 
 Otherwise, if $\dip$ has mass in $c$, then 
  \begin{align*}
  \cB(c) = \cB(c+) + \linspan\lbrace z\mapsto z\phi(z,c) \rbrace,
 \end{align*} 
 and the norms of $\cB(c)$ and $\cB(c+)$ coincide (on their intersection) by~\eqref{eqndBemb} and~\eqref{eqndBembplus}.
 Since the respective reproducing kernels differ, this guarantees that the inclusion is strict.  
 Finally, the orthogonality in~\eqref{eqnBcBcplus} follows from \cite[Theorem~I.6]{ar50}.   

 \ref{itemdBincl}
 In view of~\eqref{eqnKint} and~\eqref{eqnKcplus}, the difference $K(\zeta,z,c_2+) - K(\zeta,z,c_1)$ is simply 
 \begin{align*}
    \frac{1}{4} \int_{c_2}^{c_1} \phi(z,x)\phi(\zeta^\ast,x) & dx + \int_{c_2}^{c_1} \phi'(z,x) \phi'(\zeta^\ast,x)dx \\
        & + \int_{(c_2,c_1)} z\phi(z,x) \zeta^\ast\phi(\zeta^\ast,x)d\dip(x), \quad\zeta,\,z\in\C. 
 \end{align*}
 Since this is a kernel function itself, the inclusion follows from \cite[Theorem~I.6]{ar50}.  
 If the intersection $(c_2,c_1)\cap\Sigma$ is empty, then an integration by parts shows that $K(\zeta,\redot,c_2+) - K(\zeta,\redot,c_1)$ 
 is a linear combination of the functions $\phi(\ledot,c_1)$ and $\phi'(\ledot,c_1)$ for every $\zeta\in\C$. 
 Since the coefficient of the latter function does not vanish for all $\zeta\in\C$ and $\phi(\ledot,c_1)\in\cB(c_1)$ by \eqref{eqndBTransOrth}, this shows
 \begin{align*}
  \cB(c_2+) = \cB(c_1) + \linspan\lbrace z\mapsto \phi'(z,c_1)\rbrace,
 \end{align*}
 in view of \cite[Theorem~I.6]{ar50}. 
 As a consequence, this implies that 
 \begin{align*}
  \cB(c_1) = \cB(c_1+)  =   \cB(c_1+\varepsilon) + \linspan\lbrace z\mapsto \phi'(z,c_1+\varepsilon)\rbrace = \cB(c_2+)
 \end{align*} 
 for all small enough $\varepsilon>0$, provided the intersection $(c_2,c_1]\cap\Sigma$ is empty. 
 
 Now suppose that the intersection $(c_2,c_1]\cap\Sigma$ is not empty. 
 Without loss of generality, we may assume that $c_1$ and $c_2$ belong to $\Sigma$. 
 This is clear if $(c_2,c_1]\cap\Sigma$ contains at least two points. 
 Otherwise, just note that we may add a point mass in $c_2$ to $\omega$ without changing $\cB(c_1)$ or $\cB(c_2+)$.
 Now if the function $\phi(\ledot,c_2)$, which belongs to $\cB(c_2+)$, belonged to $\cB(c_1)$, then it would be a scalar multiple of $\phi(\ledot,c_1)$ since both are orthogonal to $\cB_0(c_1)$ in $\cB(c_1)$ by Theorem~\ref{thmdBemb} and~\eqref{eqndBembplus}.  
  In view of Proposition~\ref{propTransPE} (also note that $\delta_{c_1}$ and $\delta_{c_2}$ belong to $\D$), one infers that $\delta_{c_1}$ and $\delta_{c_2}$ are linearly dependent as well in this case, which gives a contradiction. 

 \ref{itemdBcont}
 In order to prove the first equality in~\eqref{eqndBcont}, note that one inclusion follows readily from \ref{itemdBincl}.
 For the converse, we use \cite[Section~I.9]{ar50} to conclude that $K(\zeta,\redot,b_n)$ converges to $K(\zeta,\redot,c+)$ in $\cB(c+)$ as $n\rightarrow\infty$ for every $\zeta\in\C$.
 In fact, the class $F_0$ in \cite[Section~I.9]{ar50} coincides with the union of all $\cB(b_n)$, $n\in\N$, and the norm is the one inherited from $\cB(c+)$, which can be deduced from Theorem~\ref{thmdBemb} and~\eqref{eqndBembplus}. 
 Now \cite[Theorem~I.9.II]{ar50} guarantees that $K(\zeta,\redot,b_n)$ converges in $\cB(c+)$, where the limit is $K(\zeta,\redot,c+)$ in view of~\eqref{eqnKint} and~\eqref{eqnKcplus}. 
 This shows that the functions $K(\zeta,\redot,c+)$ belong to the closure of the union on the left-hand side of~\eqref{eqndBcont} for every $\zeta\in\C$.
  
 The second equality in~\eqref{eqndBcont} simply follows from~\eqref{eqnKint} and \cite[Theorem~I.9.I]{ar50}, where the necessary bound on norms is guaranteed by~\eqref{eqndBemb} in Theorem~\ref{thmdBemb}. 

 \ref{itemdBempty}
  If $F$ belongs to $\cB(c+)$ for some $c\in\Sigma$, then using Theorem~\ref{thmdBemb} one gets 
  \begin{align*}
   |F(\zeta)|^2 & \leq K(\zeta,\zeta,c+) \|F\|_{\cB(c+)}^2 \leq 2 K(\zeta,\zeta,c+) \|F\|_{\Lmu}^2, \quad \zeta\in\C. 
  \end{align*}
  This immediately gives~\eqref{eqndBempty} if $\sup\Sigma=\infty$.
 Otherwise, the claim follows readily from~\ref{itemdBincl} alone (also note that $\sup\Sigma\in\Sigma$). 
 Moreover, in this case 
 \begin{align*}
   K(\zeta,z,\sup\Sigma+) = \frac{1}{2} \phi(z,\sup\Sigma) \phi(\zeta,\sup\Sigma)^\ast, \quad \zeta,\,z\in\C, 
 \end{align*}
 which shows that $\cB(\sup\Sigma+)$ is one dimensional. 

 \ref{itemdBdense}
  Since the image of functions in $\HR$ with compact support under $\F$ is dense in $\Lmu$, we infer from Lemma~\ref{lemdBTrans} and upon comparing~\eqref{eqnfUc} with~\eqref{eqnFhat} that \begin{align*}
   \overline{\bigcup_{c\,\in\,\R} \cB(c)} = \Lmu. 
  \end{align*}
  In view of~\ref{itemdBincl}, this immediately gives~\eqref{eqndBdense} if $\inf\Sigma=-\infty$. 
  Otherwise, the claim follows readily from~\ref{itemdBincl} alone (also note that $\inf\Sigma\in\Sigma$). 
  Moreover, in this case 
  \begin{align*}
   \cB(\inf\Sigma) \,\dot{+}\, \linspan\lbrace z\mapsto\phi'(z,\inf\Sigma)\rbrace
  \end{align*}
  is dense in $\Lmu$ by~\ref{itemdBincl}, and hence coincides with $\Lmu$ since it is closed. 
 \end{proof}

 Roughly speaking, the preceding result shows that the de Branges spaces $\cB(c)$ form a descending chain of closed subspaces of $\Lmu$ which is complete in some sense. 
 More precisely, these properties, together with de Branges' subspace ordering theorem, will allow us to conclude that every other de Branges space which is homeomorphically embedded into $\Lmu$ as well, already appears somewhere in our chain of de Branges spaces. 
 In this respect, let us finally also mention that it is possible to read off the base point of a de Branges space in the sense that  
 \begin{align}\label{eqndBcenc}
   \sup\left\lbrace |F(0)|^2 \,\left|\, F\in\cB(c),\; \|F\|_{\Lmu}=1 \right.\right\rbrace = \E^{-c}, \quad c\in\Sigma. 
 \end{align}
 which follows essentially from~\eqref{eqndBnormOC} and the fact that the embedding in Theorem~\ref{thmdBemb} preserves orthogonality. 
 Also note that by \eqref{eq:6.05} and Theorem~\ref{thmdBemb} one has 
 \begin{align*}
  F(0) = \spr{F}{2 K(0,\redot,c)}_{\Lmu}, \quad F\in\cB(c),~c\in\Sigma. 
 \end{align*}
 The same equality also holds true if we replace the space $\cB(c)$ with the space $\cB(c+)$.

\subsection{Inverse uniqueness results}

 After having gathered all necessary prerequisites, we are now able to prove several inverse uniqueness results for our spectral problem. 
 To this end, let $\tilde{\omega}$ be another real-valued Borel measure on $\R$ and $\tilde{\dip}$ be a non-negative Borel measure on $\R$. 
 All quantities corresponding to these coefficients will be denoted in an obvious way with an additional tilde. 
 
 We will first consider an inverse problem on the whole line under the additional assumption of Hypothesis~\ref{hypRESol}. 
 More precisely, we will show when the spectral measure introduced in Subsection~\ref{subsST} uniquely determines both coefficients. 

\begin{theorem}\label{thmIU}
 Suppose that Hypothesis~\ref{hypRESol} holds such that the quotient
  \begin{align}\label{eqnIUquot}
  \tilde{\phi}(z,c) \phi(z,c)^{-1}, \quad z\in\H, 
 \end{align}
 is of bounded type\footnote{A function is of bounded type if it can be written as the quotient of bounded analytic functions.} 
 in the open upper complex half-plane $\H$ for some $c\in\R$. 
 If the corresponding spectral measures $\mu$ and $\tilde{\mu}$ are equal, then $\omega=\tilde{\omega}$ and $\dip=\tilde{\dip}$.
\end{theorem}

\begin{proof}
 First of all note that the support $\Sigma$ of $|\omega|+\dip$ is empty if and only if the spectrum $\sigma(\T)$ is empty, in view of~\eqref{eqnCDomT}. 
 Therefore, the set $\Sigma$ is empty if and only if $\tilde{\Sigma}$ is empty as well. 
 Obviously, the claimed identity holds in this case. 
 
 After having dealt with this trivial case, we may suppose that the sets $\Sigma$ and $\tilde{\Sigma}$ are not empty and fix some point $c\in\Sigma$. 
 From Theorem~\ref{thmdBemb} and \cite[Theorem~A.1]{LeftDefiniteSL} (which is an immediate consequence of de Branges' subspace ordering theorem \cite[Theorem~35]{dB68}), we infer that for every $\tilde{x}\in\tilde{\Sigma}$ the space $\cB(c)$ is contained in $\tilde{\cB }(\tilde{x})$ or that $\tilde{\cB }(\tilde{x})$ is contained in $\cB(c)$. 
 Here, one should also mention that the quotient  
 \begin{align}\label{eqnQuoSplit}
   \frac{\tilde{E}(z,\tilde{x})}{E(z,c)} = \frac{\tilde{E}(z,\tilde{x})}{\tilde{\phi}(z,\tilde{x})} \cdot \frac{\tilde{\phi}(z,\tilde{x})}{\tilde{\phi}(z,c)} \cdot \frac{\tilde{\phi}(z,c)}{\phi(z,c)} \cdot \frac{\phi(z,c)}{E(z,c)}, \quad z\in\H, 
 \end{align}
  is of bounded type in $\H$, as all quotients in the factorization are as well. 
  More precisely, the first and the last factor are of bounded type in $\H$ by our definition of de Branges spaces since every function in the Hardy space $H^2(\H)$ is of bounded type \cite[Corollary~5.17]{roro94}.
  Furthermore, this means that the quotient of two functions in a de Branges space is of bounded type in $\H$, which guarantees that the second factor is as well in view of Proposition~\ref{propdBprop}~\ref{itemdBincl}. 
  In much the same way as above, one also concludes from~\eqref{eqndBembplus} that the space $\cB(c)$ is contained in $\tilde{\cB }(\tilde{x}+)$ or that $\tilde{\cB }(\tilde{x}+)$ is contained in $\cB(c)$.
  Now we introduce the sets 
\begin{align*}
 \Lambda_- & = \lbrace \tilde{x}\in\tilde{\Sigma} \,|\,  \tilde{\cB }(\tilde{x}) \supseteq \cB(c) \rbrace, &
 \Lambda_+ & = \lbrace \tilde{x}\in\tilde{\Sigma} \,|\,  \cB(c) \supseteq \tilde{\cB }(\tilde{x}+) \rbrace, 
\end{align*} 
 and note that both of them are not empty. 
 In fact, if the set $\Lambda_+$ was empty, then we would have $\cB(c) \subseteq \tilde{\cB }(\tilde{x}+)$ for every $\tilde{x}\in\tilde{\Sigma}$ and hence 
 \begin{align*}
  \cB(c) \subseteq \bigcap_{\tilde{x}\in\tilde{\Sigma}} \tilde{\cB }(\tilde{x}+) = \begin{cases} \lbrace 0\rbrace, & \sup\tilde{\Sigma} = \infty, \\ \tilde{\cB }(\sup\tilde{\Sigma}+), & \sup\tilde{\Sigma} < \infty,  \end{cases}
 \end{align*}
 in view of Proposition~\ref{propdBprop}~\ref{itemdBempty}. 
 Since $\cB(c)$ contains nonzero functions, we would obtain the contradiction that $\sup\tilde{\Sigma}<\infty$ and $\cB(c) = \tilde{\cB }(\sup\tilde{\Sigma}+)$. 
 Likewise, if the set $\Lambda_-$ was empty, then we would have $\tilde{\cB }(\tilde{x})\subseteq \cB(c)$ for every $\tilde{x}\in\tilde{\Sigma}$ and hence 
 \begin{align*}
  \Lmu \supseteq \cB(c) \supseteq \overline{\bigcup_{\tilde{x}\in\tilde{\Sigma}} \tilde{\cB }(\tilde{x})} = \begin{cases} \Lmu, & \inf\tilde{\Sigma} = -\infty, \\ \tilde{\cB }(\inf\tilde{\Sigma}), & \inf\tilde{\Sigma}>-\infty,  \end{cases} 
 \end{align*} 
 in view of Proposition~\ref{propdBprop}~\ref{itemdBdense}. 
 Since $\cB(c)$ has at least codimension one in $\Lmu$, we would obtain the contradiction that $\inf\tilde{\Sigma}>-\infty$ and $\cB(c) = \tilde{\cB}(\inf\tilde{\Sigma})$. 
  
 Since the set $\Lambda_+$ lies to the right of $\Lambda_-$, we infer that the quantities $\min\Lambda_+$ and $\max\Lambda_-$ are finite with $\max\Lambda_- \leq \min\Lambda_+$.  
 The fact that the maximum and the minimum are attained follows from Proposition~\ref{propdBprop}~\ref{itemdBcont}. 
 In particular, this guarantees the inclusions $\tilde{\cB }(\max\Lambda_-) \supseteq \cB(c) \supseteq \tilde{\cB }(\min\Lambda_++)$. 
 Moreover, if $\max\Lambda_-$ and $\min\Lambda_+$ do not coincide, then we also have the inclusions 
 \begin{align*}
  \tilde{\cB }(\max\Lambda_-) \supseteq \tilde{\cB }(\max\Lambda_-+) \supseteq \tilde{\cB }(\min\Lambda_+) \supseteq \tilde{\cB }(\min\Lambda_++),
 \end{align*}
 all of which differ at most by one dimension. 
 In fact, this follows from Proposition~\ref{propdBprop}~\ref{itemdBincl} since the intersection $(\max\Lambda_-,\min\Lambda_+)\cap\tilde{\Sigma}$ is empty in this case. 
 Consequently, the space $\cB(c)$ has to coincide with (at least) one of these spaces. 
 
 In conclusion, until now we showed that for every $c\in\Sigma$, there is some $\tilde{c}\in\tilde{\Sigma}$ such that $\tilde{\cB }(\tilde{c}) \supseteq \cB(c) \supseteq \tilde{\cB }(\tilde{c}+)$. 
 Taking equation~\eqref{eqndBcenc} into account, from this inclusion we actually conclude that $\tilde{c}=c$ and therefore also $\Sigma\subseteq\tilde{\Sigma}$.  
 Of course, due to symmetry reasons, we even have $\Sigma=\tilde{\Sigma}$ and also $\cB(c) \supseteq \tilde{\cB }(c) \supseteq \cB(c+)$. 
 This finally shows that $\cB(c) = \tilde{\cB }(c)$, including norms by Theorem~\ref{thmdBemb}.  

 As a consequence, we can employ~\eqref{eqndBTransOrth} to conclude that 
 \begin{align}\label{eqnPhiEq}
  \phi(z,x) = 2\E^{\frac{x}{2}} K(0,z,x) = 2 \E^{\frac{x}{2}} \tilde{K}(0,z,x) =  \tilde{\phi}(z,x), \quad z\in\C,
 \end{align}
 for all $x\in\Sigma$.
 Moreover, if $(a,b)$ is a gap of $\Sigma$, that is, whenever $a$, $b\in\Sigma$ but the intersection $(a,b)\cap\Sigma$ is empty, then for every $z\in\C$, the difference $\phi(z,\redot)-\tilde{\phi}(z,\redot)$ is a solution of the differential equation
 \begin{align}\label{eqnDEhozzero}
  -f'' + \frac{1}{4}f = 0
 \end{align}
  on $(a,b)$ which vanishes on the boundary of the gap. 
 But this guarantees that the solution vanishes on the whole gap and we infer that~\eqref{eqnPhiEq} holds for all $x$ in the convex hull of $\Sigma$.
 Now if $\overline{c}=\sup\Sigma$ is finite, then we have 
 \begin{align*}
  \phi(z,x) & = \phi(z,\overline{c}) \E^{-\frac{x-\overline{c}}{2}} = \tilde{\phi}(z,\overline{c}) \E^{-\frac{x-\overline{c}}{2}} = \tilde{\phi}(z,x), \quad x\geq\overline{c}. 
 \end{align*}
 On the other side, if $\underline{c}=\inf\Sigma$ is finite, then we get from Theorem~\ref{thmSARes} that 
 \begin{align*}
  \left(\frac{1}{2}-\frac{\phi'(z,\underline{c})}{\phi(z,\underline{c})}\right)^{-1} - 1 = z \spr{(\T-z)^{-1}\delta_{\underline{c}}}{\delta_{\underline{c}}}_{\HR} = \int_\R \frac{z}{\lambda-z}dE_{\delta_{\underline{c}},\delta_{\underline{c}}}(\lambda), \quad z\in\C\backslash\R. 
 \end{align*}
 From this and Lemma~\ref{lemSM}, we infer that $\phi'(z,\underline{c}) = \tilde{\phi}'(z,\underline{c})$ for every $z\in\C$.   
 But this guarantees that~\eqref{eqnPhiEq} holds for $x\leq\underline{c}$ as both sides are solutions of the differential equation~\eqref{eqnDEhozzero} to the left of $\underline{c}$ with the same boundary values at $\underline{c}$. 
 In any case, we finally conclude that~\eqref{eqnPhiEq} holds for all $x\in\R$.

  It remains to show that it is possible to read off the coefficients from the real entire solution $\phi$. 
  To this end, fix some $a$, $b\in\R$ with $a<b$ and note that for every $z\in\C$ we obtain from the differential equation~\eqref{eqnDEho} that   
 \begin{align*}
  \phi'(z,a) - \phi'(z,b) + \frac{1}{4} \int_a^b \phi(z,x)dx = z \int_a^b \phi(z,x) d\omega(x) + z^2 \int_a^b \phi(z,x) d\dip(x).  
 \end{align*}
 Since a similar equation holds for the second spectral problem as well, we get 
 \begin{align*}
   \int_a^b \phi(z,x) d\omega(x) + z \int_a^b \phi(z,x) d\dip(x) =  \int_a^b \phi(z,x) d\tilde{\omega}(x) + z \int_a^b \phi(z,x) d\tilde{\dip}(x), 
 \end{align*}
 also employing that~\eqref{eqnPhiEq} holds for all $x\in\R$. 
 Evaluating at zero, we conclude that $\omega = \tilde{\omega}$ since the points $a$ and $b$ were arbitrary. 
 Taking this into account, dividing by $z$ and evaluating at zero again, we finally end up with $\dip=\tilde{\dip}$ as well. 
\end{proof}

The assumption on the quotient in~\eqref{eqnIUquot} being of bounded type can be inconvenient for applications. 
 One way to verify it is provided by a theorem of Krein~\cite[Theorem~6.17]{roro94}, \cite[Section~16.1]{le96}, which states that an entire function is of bounded type in $\H$ if it belongs to the Cartwright class. 
 This means that the assumption on the quotient in~\eqref{eqnIUquot} holds if $\phi(\ledot,c)$ and $\tilde{\phi}(\ledot,c)$ belong to the Cartwright class for some $c\in\R$. 
 However, we can also state a variant of Theorem~\ref{thmIU} under somewhat different prerequisites on the real entire solutions $\phi$ and $\tilde{\phi}$.

\begin{corollary}
 Suppose that Hypothesis~\ref{hypRESol} holds such that the entire functions $E(\ledot,c)$ and $\tilde{E}(\ledot,\tilde{c})$ 
  are of exponential type zero for some $c\in\Sigma$ and $\tilde{c}\in\tilde{\Sigma}$. 
 If the corresponding spectral measures $\mu$ and $\tilde{\mu}$ are equal, then $\omega=\tilde{\omega}$ and $\dip=\tilde{\dip}$.
\end{corollary}

\begin{proof}
 By a  variant of de Branges' subspace ordering theorem \cite[Theorem~A.2]{LeftDefiniteSL}, \cite{ko76}, we infer from Theorem~\ref{thmdBemb} that the space $\cB(c)$ is contained in $\tilde{\cB}(\tilde{c})$ or that $\tilde{\cB}(\tilde{c})$ is contained in $\cB(c)$. 
 Similarly as for~\eqref{eqnQuoSplit}, one sees that the quotient 
 \begin{align*}
  \frac{\tilde{\phi}(z,c)}{\phi(z,c)} = \frac{\tilde{\phi}(z,c)}{\tilde{\phi}(z,\tilde{c})} \cdot \frac{\tilde{\phi}(z,\tilde{c})}{\phi(z,c)}, \quad z\in\H, 
 \end{align*}
 is of bounded type in $\H$ and it remains to apply Theorem~\ref{thmIU}. 
\end{proof}


Next, we will show that the semi-axis Weyl--Titchmarsh functions introduced in Subsection~\ref{subsecWTpm} uniquely determine the coefficients on the corresponding semi-axis.
  
\begin{theorem}
 Fix $c\in\R$ and let $\gamma$, $\tilde{\gamma}\in[0,\pi)$.  
 If the corresponding Weyl--Titchmarsh functions $m_{\gamma,\pm}$ and $\tilde{m}_{\tilde{\gamma},\pm}$ are equal, then $\gamma=\tilde{\gamma}$ as well as $\omega = \tilde{\omega}$ and $\dip = \tilde{\dip}$ on the semi-axis $J_\pm$. 
\end{theorem}

\begin{proof}
 If $m_{\gamma,\pm}$ and $\tilde{m}_{\tilde{\gamma},\pm}$ are equal, then we infer from Corollary~\ref{corWTasympm} that $\gamma=\tilde{\gamma}$. 
 In particular, this guarantees that $m_{0,\pm}$ and $\tilde{m}_{0,\pm}$ are equal as well by~\eqref{eq:wf_pm}. 
 
 Firstly, we consider the case of the left semi-axis $J_-$, that is, when $m_{0,-}$ and $\tilde{m}_{0,-}$ are equal. 
 Without loss of generality, we may assume that the measures $|\omega|+\dip$ and $|\tilde{\omega}|+\tilde{\dip}$ vanish on $J_+$. 
 As a consequence, we may choose real entire fundamental systems as in Hypothesis~\ref{hypRESol} and Lemma~\ref{lemTheta} such that 
 \begin{align*}
    \phi(z,x) & = \tilde{\phi}(z,x)  = \E^{-\frac{x}{2}}, & \theta(z,x) & = \tilde{\theta}(z,x) = \E^{\frac{x}{2}}, 
 \end{align*}
 for all $x\geq c$ and $z\in\C$. 
 In view of~\eqref{eqn_mpsi} and~\eqref{eq:wf_pm}, the corresponding singular Weyl--Titchmarsh functions are then related via 
  \begin{align*}
  M(z) =  \E^c \frac{1+2z\,m_{0,-}(z)}{1-2z\,m_{0,-}(z)} = \E^c \frac{1+2z\,\tilde{m}_{0,-}(z)}{1-2z\,\tilde{m}_{0,-}(z)} = \tilde{M}(z), \quad z\in\C\backslash\R. 
 \end{align*}
 But this shows that the corresponding spectral measures $\mu$ and $\tilde{\mu}$ are equal by~\eqref{eqndefmu}. 
 Upon noting that the quotient in~\eqref{eqnIUquot} is constant, it remains to apply Theorem~\ref{thmIU} to conclude that $\omega=\tilde{\omega}$ and $\dip=\tilde{\dip}$ on $J_-$.
 
 Secondly, we turn to the case of the right semi-axis $J_+$. 
 By a reflection argument, this case can be reduced to concluding that $\omega=\tilde{\omega}$ and $\dip=\tilde{\dip}$ on $(-\infty,c]$ from 
 \begin{align*}
   m_{0,-}(z) + \omega(\lbrace c\rbrace) + z\,\dip(\lbrace c\rbrace) = \tilde{m}_{0,-}(z) + \tilde{\omega}(\lbrace c\rbrace) + z\,\tilde{\dip}(\lbrace c\rbrace) , \quad z\in\C\backslash\R. 
 \end{align*}
 However, this can be done in much the same manner as above. 
\end{proof}

Note that in general one can not read off the position of the interior point $c$ from our semi-axis Weyl--Titchmarsh function $m_{\gamma,\pm}$.  
For example, simple counterexamples are provided by cases in which the coefficients $\omega$ and $\dip$ are periodic. 

As a final remark, let us mention that it is also possible to obtain inverse uniqueness results for our spectral problem on bounded intervals. 
Employing the methods presented in Section~\ref{secIST}, it is possible to show that the Weyl--Titchmarsh function $m_{\alpha,\beta}$ introduced in Subsection~\ref{subsecWTbound} uniquely determines the coefficients $\omega$ and $\dip$ on the corresponding bounded interval $[a,b)$. 
Based on this, one also obtains classical inverse uniqueness results in terms of two discrete spectra in the usual manner.

\end{document}